\documentclass[11pt,a4paper,reqno]{amsart}
\pdfoutput=1
\usepackage{amsmath,amsthm,latexsym,amssymb,wasysym}
\usepackage{numprint}
\usepackage[margin=1.0in,tmargin=0.9in,bmargin=0.80in]{geometry}
\usepackage{hyperref}
\usepackage{bbm}
\usepackage{xargs}
\usepackage{enumerate}
\usepackage{enumitem}
\usepackage{subcaption}
\setlist[enumerate]{label={\upshape(\roman*)}}
\usepackage[textsize=footnotesize]{todonotes}
\usepackage[flushleft]{threeparttable}
\usepackage{multirow}
\usepackage{longtable}
\usepackage{times}
\usepackage{microtype}
\usepackage[numbers]{natbib}
\usepackage{euscript}

\usepackage{amsmath,amsfonts}
\usepackage{bm}
\usepackage{dsfont}

\def\eqref#1{equation~\ref{#1}}

\def\1{\bm{1}}

\def\rmI{{\mathbf{I}}}

\DeclareMathAlphabet{\mathsfit}{\encodingdefault}{\sfdefault}{m}{sl}
\SetMathAlphabet{\mathsfit}{bold}{\encodingdefault}{\sfdefault}{bx}{n}

\def\sR{{\mathbb{R}}}

\newcommand{\E}{\mathbb{E}}

\newcommand{\Cov}{\mathrm{Cov}}

\def\paref#1{(\ref{#1})}
\def\Id{\mathds{1}}
\newcommand{\diff}{\mathrm{d}}

\npthousandsep{,}
\definecolor{bred}{rgb}{0.8,0,0}
\hypersetup{colorlinks,linkcolor={blue},citecolor={bred},urlcolor={blue}}

\newtheorem{theorem}{Theorem}[section]
\newtheorem{proposition}[theorem]{Proposition}
\newtheorem{lemma}[theorem]{Lemma}
\newtheorem{corollary}[theorem]{Corollary}
\newtheorem{remark}[theorem]{Remark}
\newtheorem{definition}[theorem]{Definition}

\newtheorem{assumption}{Assumption}

\begin{document}

\title[]{Lipschitz Transport Maps via the F{\"o}llmer Flow}

\author[Y. Dai]{Yin Dai}
\author[Y. Gao]{Yuan Gao}
\author[J. Huang]{Jian Huang}
\author[Y.L. Jiao]{Yuling Jiao}
\author[L.C. Kang]{Lican Kang}
\author[J. Liu]{Jin Liu}

\address{School of Mathematics and Statistics, Wuhan University, Wuhan, 430072, China}

\address{Department of Applied Mathematics, The Hong Kong Polytechnic University, Hong Kong, China}

\address{Department of Applied Mathematics, The Hong Kong Polytechnic University, Hong Kong, China}

\address{School of Mathematics and Statistics, Wuhan University, Wuhan, 430072, China}
\email{yulingjiaomath@whu.edu.cn}

\address{School of Mathematics and Statistics, Wuhan University, Wuhan, 430072, China}

\address{School of Data Science, Chinese University of Hong Kong, Shenzhen, China}

\thanks{${}^{\ddagger}$ Authors are listed in alphabetical order.}

\begin{abstract}
Inspired by the construction of the F{\"o}llmer process \cite{follmer1988random}, we construct a unit-time flow on the Euclidean space, termed the F{\"o}llmer flow, whose flow map at time 1 pushes forward a standard Gaussian measure onto a general target measure. We study the well-posedness of the F{\"o}llmer flow and establish the Lipschitz property of the flow map at time 1. We apply the Lipschitz mapping to several rich classes of probability measures on deriving dimension-free functional inequalities and concentration inequalities for the empirical measure.
\end{abstract}

\keywords{Lipschitz transport maps, functional inequalities, empirical measures, Gaussian mixtures.}

\maketitle

\section{Introduction}

Functional inequalities, such as Poincar{\'e}-type and Sobolev-type inequalities, are fundamental tools in studying sampling algorithms 
\cite{chewi2020exponential,li2020riemannian,lehec2021langevin, chewi2022analysis,lu2022explicit, cattiaux2022functional,andrieu2022comparison,andrieu2022poincar}, stochastic optimization \cite{raginsky2017non,xu2018global,kinoshita2022improved,li2022sharp}, and score-based generative modeling 
\cite{block2020generative,lee2022convergence,koehler2022statistical,wibisono2022convergence} in machine learning, statistics, and applied probability. Establishing functional inequalities with dimension-free constants has attracted widespread attention in various fields of mathematics like probability, geometry, and analysis \cite{wang2016functional,bardet2018functional, courtade2020bounds,chen2021dimension,cattiaux2022functional, lee2018kannan,chen2021almost,klartag2022bourgain, jambulapati2022slightly,conforti2022weak, klartag2023logarithmic}. 
The seminal work \cite{bakry1985diffusions} introduces the Bakry-{\'E}mery criterion to verify dimension-free Poincar{\'e} and log-Sobolev inequalities for strongly log-concave measures. 
Holley-Stroock perturbation principle \cite{holley1987logarithmic} implies any bounded perturbation of a strongly log-concave measure satisfies dimension-free Poincar{\'e} and log-Sobolev inequalities. Up to now, it remains an attractive topic to explore general classes of probability measures that satisfy dimension-free functional inequalities.

One of the basic principles for proving functional inequalities regarding probability measures on the Euclidean space is Lipschitz changes of variables between a source and a target probability measures \cite{caffarelli2000monotonicity, otto2000generalization, bobkov2010perturbations, kim2012generalization, colombo2017lipschitz, fathi2020proof, klartag2021spectral, mikulincer2021brownian, mikulincer2022lipschitz, neeman2022lipschitz, chewi2022entropic, shenfeld2022exact}. 
For instance, provided that $\mu$ and $\nu$ are Borel probability measures defined on $\mathbb R^d$, we would seek a Lipschitz transport map $\varphi: \mathbb R^d \rightarrow \mathbb R^d$ such that $\nu$ can be represented as a push-forward measure under $\varphi$, namely $\nu=\mu \circ \varphi^{-1}$. $\mu$ is transported onto $\nu$ in the sense that for every Borel set $B \subseteq \mathbb R^d$, $\nu(B)=\mu(\varphi^{-1}(B) )$. The Lipschitz nature of transport map $\varphi$ plays a profound impact on transferring desirable analytic results from the source measure $\mu$ to the target measure $\nu$. However, existence of such Lipschitz transport maps is not guaranteed unless proper convexity restrictions are placed on the measures. Our main goal is to extend quantitative Lipschitz regularity of transport maps to measures that are not necessarily strongly log-concave.

Let us first recall the celebrated Caffarelli's contraction theorem \cite[Theorem 2]{caffarelli2000monotonicity}.
Let $\mu (\diff x) = \exp(-U(x)) \diff x$ and $\nu (\diff x) = \exp(-W(x)) \diff x$ be two probability measures defined on $\mathbb R^d$ with $U, W \in C^2(\sR^d)$. Suppose that $\nabla^2 U(x) \preceq \beta \mathbf{I}_d$ and $\nabla^2 W(x) \succeq \alpha \mathbf{I}_d \succ 0$. 
Then the optimal transport map $\varphi_{\mathrm{opt}}:=\nabla \psi$ from $\mu$ to $\nu$ is $\sqrt{\beta/\alpha}$-Lipschitz, where $\psi: \mathbb R^d \rightarrow \mathbb{R}$ is the convex Brenier potential. 
The optimal transport map $\varphi_{\mathrm{opt}}: \mathbb R^d \rightarrow \mathbb R^d$ pushes forward $\mu$ onto $\nu$ in the sense that $\nu = \mu \circ ( \varphi_{\mathrm{opt}} )^{-1}$.
In particular, if $\gamma_d$ is the standard Gaussian measure on $\mathbb R^d$ and probability measure $\mu$ has a log-concave density with respect to $\gamma_d$,
then there exists a 1-Lipschitz map $\varphi_{\mathrm{opt}}$ such that $\nu = \gamma_d \circ \left(\varphi_{\mathrm{opt}} \right)^{-1}$.
The 1-Lipschitz transport map $\varphi_{\mathrm{opt}}$ enables dimension-free functional inequalities to be transferred from $\gamma_d$ to $\nu$.
Recently, \cite{mikulincer2021brownian} defines a Brownian transport map, based on the F{\"o}llmer process defined in Definition \ref{def:follmer-sde}, that transports the infinite-dimensional Wiener measure onto probability measures on the Euclidean space. 
Lipschitz properties of the Brownian transport map are investigated extensively while no analogous results for optimal transport maps are known.
\cite{neeman2022lipschitz} and \cite{mikulincer2022lipschitz} also utilize a Lipschitz transport map along the reverse heat flow, which previously appears in \cite{otto2000generalization} and is further studied by \cite{kim2012generalization}, to establish functional inequalities, perform eigenvalues comparisons, and study domination of distribution functions.
On the Caffarelli’s contraction theorem, \cite{fathi2020proof} and \cite{chewi2022entropic} provide new proofs using the entropic interpolation between the source and the target measures.

In this work, we construct a flow over the unit time interval on the Euclidean space, named the F{\"o}llmer flow as in Definition \ref{def:follmer-flow} and Theorem \ref{main-thm1}. 
Our construction is greatly enlightened by F{\"o}llmer's derivation of the F{\"o}llmer process.
Then we define and analyze a new transport map, along the F{\"o}llmer flow, which pushes forward the standard Gaussian measure to a general measure satisfying mild regularity assumptions (see Assumptions \ref{condn:well-defined}, \ref{condn:semi-log-convex} and \ref{condn:semi-log-concave}).
The well-posedness of the F{\"o}llmer flow and the Lipschitz property of its flow map at time 1 are rigorously investigated under these regularity assumptions. 
By virtue of the Lipschitz changes of variables principle, we prove dimension-free $\Psi$-Sobolev inequalities, isoperimetric inequalities, $q$-Poincar{\'e} inequalities and sharp non-asymptotic  concentration bounds for the empirical measure. 
Furthermore, we shall emphasize that both the F{\"o}llmer flow and its flow map possess much computational flexibility in terms of the analytic expression of its velocity field, which we believe may be of independent interest, to develop sampling algorithms and generative models with theoretical guarantees. 

\subsection{Related work}
The work is notably relevant to the Brownian transport map built upon the F{\"o}llmer process \cite{mikulincer2021brownian} and the transport map defined via the reverse heat flow \cite{otto2000generalization, kim2012generalization, neeman2022lipschitz, mikulincer2022lipschitz}. 
The Brownian transport map, acquired from a strong solution of the F{\"o}llmer process, pushes forward the Wiener measure onto probability measures on the Euclidean space. 
The infinite-dimensional nature of the Brownian transport map is quite different from that of the F{\"o}llmer flow which is defined on the finite-dimensional Euclidean space. 
To produce the randomness within the target measure, the Brownian transport map leverages the randomness of the path while the F{\"o}llmer flow makes use of the randomness delivered by the source measure.
Meanwhile, \cite{mikulincer2022lipschitz} studies a transport map along the reverse heat flow from the standard Gaussian measure to a target measure, constructed by \cite{otto2000generalization} and \cite{kim2012generalization}, as well as its Lipschitz property. 
The transport map investigated in the work shares a similar Lipschitz property with the transport map associated with the reverse heat flow.
Nonetheless, the transport map investigated by \cite{kim2012generalization} and \cite{mikulincer2022lipschitz} is deduced via a limiting argument, thus has no explicit expression.
Under Assumptions \ref{condn:well-defined}, \ref{condn:semi-log-convex} and \ref{condn:semi-log-concave}, our considered transport map could be expressed as the flow map of the well-posed F{\"o}llmer flow at time $t = 1$ in a simple and explicit form. 
Towards connections between the flows, the F{\"o}llmer flow over time interval $[0, 1)$ (without time $1$) is equivalent to the reverse heat flow through a deterministic change of time, as revealed in Lemmas \ref{lm:sde-time-change} and \ref{lm:ode-time-change}. 
Technically, the equivalence cannot ensure well-posedness of the F{\"o}llmer flow at time $1$. We 
explicitly extend the flow to time $1$ by deriving a uniform lower bound on the Jacobian matrix of the velocity field via the Cram{\'e}r-Rao inequality.
Additionally, it is worth mentioning that \cite{albergo2023building} and \cite{albergo2023stochastic} introduce a unit-time normalizing flow, relevant to the F{\"o}llmer flow, from the perspective of stochastic interpolation between the Gaussian measure and a target measure. 
Nonetheless, the well-posedness of this normalizing flow is not studied in the scope of their work.

\subsection{Notations}
For any integer $d\geq 1$, the Borel $\sigma$-algebra of $\mathbb R^d$ is denoted by $\mathcal B(\mathbb R^d)$.
For $x,y\in \mathbb R^d$, define $\left<x, y\right> := \sum_{i=1}^d x_iy_i$ and the Euclidean norm $|x|:=\left<x,x \right>^{1/2}$. Denote by $\mathbb S^{d-1} :=\{x\in \mathbb R^d: |x|=1 \}$. The operator norm of a matrix $M \in \mathbb R^{m\times n}$ is denoted by $\|M \|_{\mathrm{op}} := \sup_{x\in \mathbb R^n, |x|=1} |Mx|$ and $M^{\top}$ is the transpose of $M$. Let $f:\mathbb R^d \rightarrow \mathbb R$ be a twice continuously differentiable function. Denote by $\nabla f, \nabla^2 f$ and $\Delta f$ the gradient of $f$, the Hessian of $f$ and the Laplacian of $f$, respectively. Let $\gamma_d$ denote the standard Gaussian measure on $\mathbb R^d$, i.e.,
$\gamma_d(\diff x) := (2\pi)^{-d/2} \exp(-|x|^2 /2) \diff x$.
Let $N (0, \rmI_d)$ stand for a $d$-dimensional Gaussian random variable with mean $0$ and covariance $\mathbf{I}_d$ being the $d\times d$ identity matrix. Moreover, we use $\phi(x)$ to denote its probability density function with respect to the Lebesgue measure.

The set of probability measures defined on a measurable space $(\mathbb R^d, \mathcal B(\mathbb R^d))$ is denoted as $\mathcal P(\mathbb R^d)$.
For any $\mathbb R^d$-valued random vector, $\mathbb E[X]$ is used to denote its expectation.
We say that $\Pi$ is a transference plan of $\mu$ and $\nu$ if it is a probability measure on $(\mathbb R^d \times \mathbb R^d, \mathcal B(\mathbb R^d) \times \mathcal B(\mathbb R^d))$ such that for any Borel set $A$ of $\mathbb R^d$, $\Pi(A \times \mathbb R^d)=\mu(A)$ and $\Pi(\mathbb R^d \times A)=\nu(A)$.
We denote $\mathcal C(\mu,\nu)$ the set of transference plans of $\mu$ and $\nu$. Furthermore, we say that a couple of $\mathbb R^d$-valued random variables $(X,Y)$ is a coupling of $\mu$ and $\nu$ if there exists $\Pi \in \mathcal C(\mu,\nu)$ such that $(X,Y)$ is distributed according to $\Pi $. For two probability measures $\mu,\nu \in \mathcal P(\mathbb R^d)$, the Wasserstein distance of order $p \geq 1$ is defined as
 \begin{equation*}
    W_p(\mu,\nu) := \inf_{\Pi \in \mathcal C(\mu,\nu)} \left(\int_{\mathbb R^d \times \mathbb R^d}  |x-y|^p \,\Pi(\diff x, \diff y) \right)^{1/p}.
 \end{equation*}
Let $\mu,\nu \in \mathcal P(\mathbb R^d)$. The relative entropy of $\nu$ with respect to $\mu$ is defined by
\begin{equation*}
H(\nu \mid \mu) =
\begin{cases}
 \int_{\mathbb R^d} \log \left( \frac{\diff \nu}{\diff \mu} \right) \nu(\diff x), & \text{if $\nu \ll \mu $, } \\
 +\infty, & \text{otherwise.} 
\end{cases}
\end{equation*}

\section{Main results}
We first present two definitions to characterize convexity properties of probability measures and some useful notations.
\begin{definition}[\cite{cattiaux2014semi, mikulincer2021brownian}]
\label{def:semi-log-concave}
A probability measure $\mu (\diff x) = \exp(-U(x)) \diff x$ is $\kappa$-semi-log-concave for some $\kappa \in \sR$ if its support $\Omega \subseteq \sR^d$ is convex and $U \in C^2(\Omega)$ satisfies
\begin{equation*}
\nabla^2 U(x) \succeq \kappa \mathbf{I}_d, \quad \forall x \in \Omega.
\end{equation*}
\end{definition}

\begin{definition}[\cite{eldan2018regularization}]
\label{def:semi-log-convex}
A probability measure $\mu (\diff x) = \exp(-U(x)) \diff x$ is $\beta$-semi-log-convex for some $\beta > 0$ if its support $\Omega \subseteq \sR^d$ is convex and $U \in C^2(\Omega)$ satisfies
\begin{equation*}
\nabla^2 U(x) \preceq \beta \mathbf{I}_d, \quad \forall x \in \Omega.
\end{equation*}
\end{definition}

Let $\nu(\diff x) = p(x) \diff x$ be a probability measure on $\sR^d$
and define an operator $(\EuScript Q_t)_{t\in [0,1]}$, acting on function $f:\mathbb R^d \rightarrow \mathbb R$ by
\begin{align*}
    \EuScript{Q}_{1-t} f (x) 
     := \int_{\mathbb R^d} \varphi^{t x, 1-t^2} (y) f(y) \diff y  = \int_{\mathbb{R}^d} f \left(t x + \sqrt{1 - t^2} z \right) \diff \gamma_d (z)
\end{align*}
where $\varphi^{tx, 1-t^2} (y)$ is the density of the $d$-dimensional Gaussian measure with mean $tx$ and covariance $(1-t^2)\mathbf{I}_d $. 

Our first result is that we construct a flow over the unit time interval, named the F{\"o}llmer flow, that pushes forward a standard Gaussian measure $\gamma_d$ to a general target measure $\nu$ at time $t=1$.
Before rigorously defining the F{\"o}llmer flow, let us specify several regularity assumptions that would ensure well-definedness and well-posedness of the F{\"o}llmer flow.

\begin{assumption} \label{condn:well-defined} 
The probability measure $\nu$ has a finite third moment and is absolutely continuous with respect to the standard Gaussian measure $\gamma_d$.
\end{assumption}

\begin{assumption} \label{condn:semi-log-convex}
The probability measure $\nu$ is $\beta$-semi-log-convex for some $\beta >0$.
\end{assumption}

\begin{assumption} \label{condn:semi-log-concave} 
Let $D := (1 / \sqrt{2}) \mathrm{diam} (\mathrm{supp} (\nu))$. The probability measure $\nu$ satisfies one or more of the following assumptions:
\begin{itemize}
\item[(i)]    $\nu$ is $\kappa$-semi-log-concave for some $\kappa > 0$ with $D \in (0, \infty]$; 
\item[(ii)]   $\nu$ is $\kappa$-semi-log-concave for some $\kappa \le 0$ with $D \in (0, \infty)$; 
\item[(iii)]  $\nu = N(0,\sigma^2 \mathbf{I}_d) * \rho$  where $\rho$ is a probability measure supported on a ball of radius $R$ on $\mathbb R^d$. 
\end{itemize}
\end{assumption}

Let us move to a formal definition of the F{\"o}llmer flow and the exhibition of its well-posedness. A complete exposition would be found in Section \ref{sec:follmer-flow}.  
\begin{definition}
\label{def:follmer-flow}
Suppose that probability measure $\nu$ satisfies Assumption \ref{condn:well-defined}. If $(X_t)_{t \in [0, 1]}$ solves the initial value problem (IVP)
\begin{equation}\label{main-ode}
    \frac{\diff X_t} {\diff t} = V(t, X_t), \quad X_0 \sim \gamma_d, \quad t \in [0,1]
\end{equation}
where the velocity field $V$ is defined by
\begin{equation}
\label{eq:vector-field}
    V(t, x) :=\frac{\nabla \log \EuScript{Q}_{1-t} r(x)}{t}, \qquad \forall t\in (0,1]
\end{equation}
with $V(0, x) :=\mathbb E_{\nu} [X], r(x) := \frac{\diff \nu}{\diff \gamma_d} (x)$. We call $(X_t)_{t \in [0, 1]}$ a F{\"o}llmer flow and $V(t, x)$ a F{\"o}llmer velocity field associated to $\nu$. 
\end{definition}

\begin{theorem}[Well-posedness] \label{main-thm1}
Suppose that Assumptions \ref{condn:well-defined}, \ref{condn:semi-log-convex} and \ref{condn:semi-log-concave} hold.
Then the F{\"o}llmer flow $(X_t)_{t \in [0, 1]}$ associated to $\nu$ is a unique solution to the IVP \paref{main-ode}.
Moreover, the push-forward measure $\gamma_d \circ (X_1^{-1}) = \nu$.
\end{theorem}

The following results show that the F{\"o}llmer flow map at time $t = 1$ is Lipschitz when the target measure satisfies either the strong log-concavity assumption or the bounded support assumption.

\begin{theorem}[Lipschitz mapping] \label{main-thm2}
 Assume that Assumptions \ref{condn:well-defined}, \ref{condn:semi-log-convex}, \ref{condn:semi-log-concave}-(i) or \ref{condn:semi-log-concave}-(ii) hold.
\begin{itemize}
    \item[(i)] If $\kappa D^2 \ge 1 $, then $X_1(x)$ is a Lipschitz mapping with constant $\tfrac{1}{\sqrt{\kappa}}$, i.e.,
    \begin{equation*}
    \|\nabla X_1(x) \|_{\mathrm{op}} \le \frac{1}{\sqrt{\kappa}}, \quad \forall x \in \sR^d.
    \end{equation*}

    \item[(ii)] If $\kappa D^2 <1 $, then $X_1(x)$ is a Lipschitz mapping with constant $\exp\left(\frac{1-\kappa D^2}{2} \right)D $, i.e.,
    \begin{equation*}
    \|\nabla X_1(x) \|_{\mathrm{op}} \le \exp\left(\frac{ 1-\kappa D^2}{2} \right) D, \quad \forall x \in \sR^d.
    \end{equation*}
\end{itemize}
\end{theorem}

\begin{theorem}[Gaussian mixtures] \label{main-thm3}
Assume that Assumptions \ref{condn:well-defined}, \ref{condn:semi-log-convex} and \ref{condn:semi-log-concave}-(iii) hold.
Then $X_1(x)$ is a Lipschitz mapping with constant $\sigma \exp \left( \frac{R^2}{2 \sigma^2} \right)$, i.e.,
\begin{equation*}
    \|\nabla X_1(x) \|_{\mathrm{op}} \leq \sigma \exp\left( \frac{R^2}{2 \sigma^2} \right), \quad \forall x \in \sR^d.
\end{equation*}
\end{theorem}

\begin{remark}
Combining $\mathrm{Lip}(X_1(x)) \leq \|\nabla X_1(x) \|_{\mathrm{op}}$ and Theorem \ref{main-thm3}, we get
    \begin{equation}
        \label{eq:ode-gradx-bd-var-gaussian-conv}
        \mathrm{Lip}(X_1(x)) \leq \sigma \exp \left( \frac{R^2}{2 \sigma^2} \right), \quad \forall x \in \sR^d.
    \end{equation}

For Gaussian mixtures, the Lipschitz constants of (\ref{eq:ode-gradx-bd-var-gaussian-conv}) are better than those provided by the Brownian transport map \cite[Theorem 1.4]{mikulincer2021brownian} and match those presented in \cite{mikulincer2022lipschitz}. Meanwhile, the Lipschitz constants of $X_1$ lead to a dimension-free logarithmic Sobolev constant and a dimension-free Poincar\'e constant
    \begin{equation}
        \label{eq:conv-lsc-pc}
        C_{\mathrm{LS}} (p) \le 2 \sigma^2 \exp \left( \frac{R^2}{\sigma^2} \right), \quad
        C_{\mathrm{P}} (p) \le \sigma^2 \exp \left( \frac{R^2}{\sigma^2} \right).
    \end{equation}

On the one hand, (\ref{eq:conv-lsc-pc}) implies a Gaussian log-Sobolev constant $2 \sigma^2$ and a Gaussian Poincar\'e constant $\sigma^2$ as $R$ goes to zero. In fact, Poincar{\'e} constant $\sigma^2$ and log-Sobolev constant $2\sigma^2$ are optimal for Gaussian measure $N(0,\sigma^2 \mathbf{I}_d)$ on $\mathbb R^d$.
On the other hand, the Poincar{\'e} constant obtained by (\ref{eq:conv-lsc-pc}) is obviously smaller than the result in \cite[Theorem 1.2]{bardet2018functional}. In fact, the upper bound of Poincar{\'e} constant for distribution $p=N(0,\sigma^2 \mathbf{I}_d) * \rho$ in \cite[Theorem 1.2]{bardet2018functional} is $\sigma^2 \exp \left(4R^2 / \sigma^2 \right)$. Similarly, the log-Sobolev constant (\ref{eq:conv-lsc-pc}) we obtained is slightly better than that in \cite[Corollary 1]{chen2021dimension}. Indeed, the upper bound of log-Sobolev constant for distribution $\nu = N(0, \sigma^2 \mathbf{I}_d) * \rho$ in \cite[Corollary 1]{chen2021dimension} is $6(4R^2 +\sigma^2) \exp \left(4R^2 / \sigma^2 \right)$. Nonetheless, it is worthwhile to remark that \cite{chen2021dimension} considers a rich class of probability measures with the convolutional structure, which leads to general results on dimension-free log-Sobolev and Poincar{\'e} inequlities.

\end{remark}


\section{The F{\"o}llmer flow and its well-posedness} \label{sec:follmer-flow}

Let us present our motivations to derive the F{\"o}llmer flow. We are largely inspired by the construction of the F{\"o}llmer process \cite{follmer1988random, lehec2013representation, eldan2018regularization, eldan2020stability}, which provides a probabilistic solution to the Schr{\"o}dinger problem \cite{schrodinger1931uber, leonard2014survey}, though our construction of the F{\"o}llmer flow is partially heuristic using a similar time-reversal argument.

\subsection{The F{\"o}llmer process}

In F{\"o}llmer's lecture notes at the {\'E}cole d'{\'E}t{\'e} de Probabilit{\'e}s de Saint-Flour in 1986 \cite{follmer1988random}, the F{\"o}llmer process is constructed with time reversal of a linear It\^o SDE under a finite relative entropy condition, which rigorously determines a Schr{\"o}dinger bridge from a source Dirac measure $\delta_0$ to a general target measure $\nu$. Let us briefly revisit F{\"o}llmer's arguments to derive such a process.

\begin{definition} [\cite{follmer1988random}]
A diffusion process $\overline{P} := \left(\overline{X}_t \right)_{t\in [0, 1]}$ starting with marginal distribution $\nu$ at time $t = 0$ and reaching $0$ at time $t = 1$ is defined by the following It\^o SDE
\begin{equation}
    \label{eq:follmer-sde-tr}
    \diff \overline{X}_t = - \frac{1} {1 - t} \overline{X}_t \diff t + \diff \overline{W}_t, \ \overline{X}_0 \sim \nu, \ t \in [0, 1)
\end{equation}
with an extended solution at time $t = 1$, i.e., $\overline{X}_1 \sim \delta_{0}$.
The transition probability distribution of (\ref{eq:follmer-sde-tr}) from $\overline{X}_0$ to $\overline{X}_t$ is given by $\overline{X}_t | \overline{X}_0 \sim N ((1-t) \overline{X}_0, t(1-t)\mathbf{I}_d )$ for every $0 \leq t < 1$.
\end{definition}

\begin{lemma} [\cite{follmer1985entropy}]
\label{lm:follmer-tr}
Suppose that the diffusion process $Q$ has finite relative entropy with respect to a standard Wiener process $W_t$ over the unit time interval, i.e., $t \in [0, 1]$. Then for almost all $t \in [0, 1]$, the logarithmic derivative of marginal density $\rho_t$ of $Q$ satisfies the duality equation $\nabla \log \rho_t (x) = b (x, t) + \overline{b} (x, 1-t)$ for almost all $x \in \sR^d$, where $b (x, t)$ and $\overline{b} (x, t)$ are drifts of diffusion process $Q$ and its time-reversed diffusion process $\overline{Q}$, respectively.
\end{lemma}

\begin{definition}[\cite{follmer1988random, lehec2013representation}]
\label{def:follmer-sde}
F{\"o}llmer process $P = (X_t)_{t \in [0, 1]}$ is defined by the It\^o SDE
\begin{equation}
    \label{eq:follmer-sde}
    \diff X_t = \nabla \log \EuScript{P}_{1-t} r (X_t) \diff t + \diff W_t, \ X_0 = 0, \ t \in [0, 1]
\end{equation}
where $W_t$ is a standard Wiener process and $\EuScript{P}_t$ is the heat semigroup defined by
$\EuScript{P}_t h(x) := \E \left[ h(x + W_t) \right]$.
Moreover, the drift $\nabla \log \EuScript{P}_{1-t} r (X_t)$ is called the F{\"o}llmer drift.
\end{definition}

\begin{remark}
According to Lemma \ref{lm:follmer-tr}, the F{\"o}llmer process $P$ can be obtained by taking the time reversal of the diffusion process $\overline{P}$ over $t \in [0, 1]$. It implies that the F{\"o}llmer drift has an alternative representation, i.e.,
for any $t \in (0, 1]$, $\nabla \log \EuScript{P}_{1-t} r (X_t) = X_t/t + \nabla \log p_t (X_t)$,
where $p_t$ is the marginal density of the F{\"o}llmer process $P$.
\end{remark}

\subsection{The F{\"o}llmer flow via time reversal}

Since $\delta_0$ is a degenerate distribution in the sense that its mass is concentrated at $0$, we consider constructing a diffusion process that starts with a marginal distribution $\nu$ and would be able to keep the nonzero variance of its marginal distribution at time $t = 1$. Let us present the constructed diffusion process first.
For any $\varepsilon \in (0,1)$, we consider a diffusion process $\left(\overline{X}_t \right)_{t\in [0, 1-\varepsilon]}$ defined by the following It\^o SDE
\begin{equation}
    \label{eq:vp-sde-tr}
    \diff \overline{X}_t = - \frac{1} {1 - t} \overline{X}_t \diff t + \sqrt{\frac{2} {1-t}} \diff \overline{W}_t, \quad \overline{X}_0 \sim \nu
\end{equation}
for all $t \in [0, 1-\varepsilon]$. By Theorem 2.1 in \cite[Chapter IX]{revuz2013continuous}, the diffusion process $\overline{X}_t$ defined in (\ref{eq:vp-sde-tr}) has a unique strong solution on $[0,1-\varepsilon]$. Moreover, the transition probability distribution of (\ref{eq:vp-sde-tr}) from $\overline{X}_0$ to $\overline{X}_t$ is given by $\overline{X}_t | \overline{X}_0 =x_0 \sim N( (1-t) x_0, \ t (2-t) \rmI_d)$
for every $t\in [0,1-\varepsilon]$.
It is a straightforward observation that the variance of $\overline{X}_{1 - \varepsilon} | \overline{X}_0$ for SDE (\ref{eq:vp-sde-tr}) will approach the identity matrix $\mathbf{I}_d$ when $\varepsilon$ is small enough.
That is why we could expect the marginal distribution of $\overline{X}_{1 - \varepsilon}$ would have a nonzero variance.
In contrast, for the time-reversed F{\"o}llmer process (\ref{eq:follmer-sde-tr}), the variance of $\overline{X}_{1 - \varepsilon} | \overline{X}_0$ will approach constant $0$ as $\varepsilon \to 0$, which indicates the variance of its marginal distribution vanishes at time $t = 1$.
However, SDE (\ref{eq:vp-sde-tr}) is not well-defined at time $t = 1$ due to unbounded drift and diffusion coefficients.
Then we leverage the fact that the marginal distribution $\overline{\mu}_t$ of the diffusion process $\overline{X}_t$ defined in (\ref{eq:vp-sde-tr}) has been determined in the sense that $\overline{X}_t \overset{d}{=} (1-t) X + \sqrt{t(2-t)} Y$ with $X \sim \nu, Y \sim \gamma_d$, and concentrate on an ODEs system sharing the same marginal distribution flow with SDE (\ref{eq:vp-sde-tr}) in order to circumvent the singularity of SDE (\ref{eq:vp-sde-tr}) at time $t = 1$.

Note that the marginal distribution flow $( \overline{\mu}_t )_{t \in [0, 1-\varepsilon]}$ of the diffusion process (\ref{eq:vp-sde-tr}) satisfies the Fokker-Planck-Kolmogorov equation in an Eulerian framework \cite{bogachev2022fokker}
\begin{equation}
    \label{eq:fke-vp-sde-tr}
    \partial_t \overline{\mu}_t = \nabla \cdot (\overline{\mu}_t V(1-t, x)) \quad \textrm{on} \ [0, 1-\varepsilon] \times \sR^d, \ \overline{\mu}_0 = \nu
\end{equation}
in the sense that $\overline{\mu}_t$  is continuous in $t$ under the weak topology, i.e.,
\begin{align*}
    \overline{\mu}_t(f)  :=\int_{\mathbb R^d} f(x) \mu_t(\diff x) =\nu(f) - \int_0^t \overline{\mu}_s \left( \left<V(1-s, \cdot), \nabla f \right> \right) \diff s
\end{align*}
for all $f\in C^{\infty}_0(\mathbb R^d)$ and the velocity field is given by
\begin{equation} \label{eq:vp-ode-eps-tr-vf}
    V(1-t, x) := \frac{1}{1-t} \left[x + S(1-t, x) \right], \quad t \in[0, 1-\varepsilon]
\end{equation}
and
\begin{equation*}
S(t, x) := \nabla \log \int_{\mathbb R^d} (2\pi (1-t^2))^{-\frac d 2} \exp\left(-\frac{|x- t y|^2}{2(1-t^2)} \right) p(y) \diff y
\end{equation*}
for all $t \in [\varepsilon, 1]$. Due to the classical Cauchy-Lipschitz theory \cite[Section 2]{ambrosio2014continuity} with a Lipschitz velocity field or the well-established Ambrosio-DiPerna-Lions theory with lower Sobolev regularity assumptions on the velocity field \cite{diperna1989ordinary,ambrosio2004transport}, we shall define a flow $( X^{*}_t )_{t \in [0, 1-\varepsilon]}$ in a Lagrangian formulation via the following ODEs system
\begin{align} \label{eq:vp-ode-eps-tr}
    \diff X^{*}_t = - V \left( 1-t, X^{*}_t\right) \diff t, \quad  X^{*}_0 \sim \nu, \quad  t \in [0, 1-\varepsilon].
\end{align}

\begin{proposition}
\label{prop:vp-ode-eps}
Assume the velocity field $V(t, x)$ satisfies $V \in L^1([\varepsilon, 1]; W^{1, \infty}_{\mathrm{loc}}(\sR^d; \sR^d))$ and $|V|/(1+|x|) \in  L^1([\varepsilon, 1]; L^{\infty}(\sR^d))$. Then the push-forward measure associated with the flow map $X^{*}_t$ satisfies $X^{*}_t \overset{d}{=} (1-t) X + \sqrt{t(2-t)} Y$ with $X \sim \nu, Y \sim \gamma_d$.
Moreover, the push-forward measure $\nu \circ ({X^{*}_{1 - \varepsilon}})^{-1}$ converges to the Gaussian measure $\gamma_d$ in the sense of Wasserstein-2 distance as $\varepsilon$ tends to zero, i.e., $W_2 (\nu \circ ({X^{*}_{1 - \varepsilon}})^{-1}, \gamma_d) \to 0$.
\end{proposition}

\begin{remark}
Suppose that the target measure $\nu$ has a finite third moment. By Lemma \ref{lem-4.1}, we can supplement the definition of velocity field $V(1-t, x)$ at time $t=1$, i.e.,
\begin{equation*}
V(0, x) := \lim_{t \downarrow 0} V(t, x) =\lim_{t \downarrow 0} \frac{x +S(t, x)}{t} =\mathbb E_{\nu}[X].
\end{equation*}
Then we extend the flow $(X^{*}_t)_{t \in [0,1)}$ to time $t =1$ such that $X^{*}_1 \sim \gamma_d$, which solves the IVP
\begin{align}
    \label{eq:vp-ode-unit-tr}
    \diff X^{*}_t = - V \left( 1-t, X^{*}_t\right) \diff t, \quad  X^{*}_0 \sim \nu, \quad  t \in [0, 1],
\end{align}
where the velocity field
\begin{equation*}
V(1-t, x) = \frac{1}{1-t} \left[ x + S(1-t, x) \right], \quad \forall t \in [0, 1)
\end{equation*}
and $V(0, x) = \E_{\nu} [X] $.
\end{remark}

In order to exploit a time-reversal argument inspired by F{\"o}llmer, it remains crucial to establish the well-posedness of a flow $( X^{*}_t )_{t \in [0, 1]}$ that solves the IVP (\ref{eq:vp-ode-unit-tr}). We proceed to study regularity properties of the velocity field $V$ on $[0, 1] \times \sR^d$ by imposing structural assumptions on the target measure $\nu$. By Theorem \ref{thm1}, we know that there exists $0 \leq \theta^{\star}_t <\infty $ such that
\begin{equation}
 \|- \nabla V(t, x) \|_{\mathrm{op}}= \| \nabla V(t, x) \|_{\mathrm{op}} \leq \theta^{\star}_{t}
\end{equation}
for any $t\in [0,1]$. Furthermore, the velocity field $-V(1-t, x)$ is smooth and with the bounded derivative for any $t\in [0,1]$ and $x\in \mathbb R^d$. Therefore, the IVP (\ref{eq:vp-ode-unit-tr}) has a unique solution and the flow map $x \mapsto X^{*}_t(x)$ is a diffeomorphism from $\mathbb R^d$ onto $\mathbb R^d$ at any time $t \in [0, 1]$. A standard time-reversal argument of ODE would yield a formal definition of the F{\"o}llmer flow.

\begin{definition}
    \label{def:follmer-flow-score}
    Suppose that probability measure $\nu$ satisfies Assumption \ref{condn:well-defined}. If $(X_t)_{t \in [0, 1]}$ solves the IVP
    \begin{equation}
    \label{eq:vp-ode-unit}
    \frac{\diff X_t} {\diff t} = V(t, X_t), \quad X_0 \sim \gamma_d, \quad t \in [0,1]
    \end{equation}
    where the velocity field
    \begin{equation*}
    V(t, x) = \frac{1}{t} \left[ x + S(t, x) \right], \quad \forall t \in (0, 1], \quad V(0, x) = \E_{\nu} [X],
    \end{equation*}
    we call $(X_t)_{t \in [0, 1]}$ a F{\"o}llmer flow and $V(t, x)$ a F{\"o}llmer velocity field associated to $\nu$. 
\end{definition}

\begin{remark}
Notice that
\begin{align*}
\EuScript{Q}_{1-t} r (x) = (2\pi)^{d/2} \exp\left(\frac{ |x |^2}{2} \right) \frac{1}{(2\pi(1-t^2))^{d/2}} \int_{\mathbb R^d} p(y)\exp\left(-\frac{|x-ty |^2 }{2(1-t^2)} \right) \diff y
\end{align*}
where $\EuScript{Q}_{1-t} r(x)$ is defined in (\ref{eq:vp-vf-op-form}). We further obtain
$\nabla \log \EuScript{Q}_{1-t} r(x) = x + S(t, x), \ \forall t \in [0,1]$.
Therefore, we have that (\ref{main-ode}) and (\ref{eq:vp-ode-unit}) are equivalent, which satisfy $X_0 \sim \gamma_d$ and $X_1 \sim \nu$.
\end{remark}

Finally, let us conclude with the well-posedness properties of the F{\"o}llmer flow, which is presented in Theorem \ref{main-thm1} and summarized below.  
\begin{theorem}[Well-posedness] \label{thm:well-posed-follmer-flow}
Suppose that Assumptions \ref{condn:well-defined}, \ref{condn:semi-log-convex} and \ref{condn:semi-log-concave} hold.
Then the F{\"o}llmer flow $(X_t)_{t \in [0, 1]}$ associated to $\nu$ is a unique solution to the IVP (\ref{eq:vp-ode-unit}).
Moreover, the push-forward measure $\gamma_d \circ (X_1^{-1}) = \nu$.
\end{theorem}

\section{Applications}
Owing to the Lipschitz transport properties proved in Theorems \ref{main-thm2} and \ref{main-thm3}, we are motivated to establish a variety of functional inequalities and concentration inequalities for several classes of probability measures on Euclidean space. 

\subsection{Dimension-free inequalities}

In this subsection, we provide dimension-free results for the $\Psi$-Sobolev inequalities. For completeness, we incorporate classical results for strongly-log-concave measures ($\kappa > 0$), which have been studied with the optimal transport maps \cite{caffarelli2000monotonicity}. Compared with \cite{mikulincer2021brownian}, we obtain that the upper bound constants of $\Psi$-Sobolev inequalities, Isoperimetric inequalities and $q$-Poincar{\'e} inequalities are the same for $\kappa D^2 \geq 1$. When $\kappa D^2 <1$, our upper bound constants to $\Psi$-Sobolev inequalities and Isoperimetric inequalities are in the same order with the results of Lemmas 5.3-5.5 in \cite{mikulincer2021brownian}. For the Gaussian mixtures case,  we obtain that the constants of these inequalities are slightly better than the result of Lemmas 5.3-5.5 in \cite{mikulincer2021brownian}.

\begin{definition}
Let $\mathcal I$ be a closed interval (not necessarily bounded) and let $\Psi: \mathcal I \rightarrow \mathbb R$ be a twice differentiable function. We say that $\Psi$ is a divergence if each of the functions $\Psi, \Psi''$ and $-1/\Psi''$ is a convex function. Given a probability measure $\nu(\diff x)=p(x) \diff x$ on $\mathbb R^d$ and a function $\zeta: \mathbb R^d \rightarrow \mathcal I$ such that $\int_{\mathbb R^d} \zeta(x) p(x) \diff x \in \mathcal I $, we define
\begin{equation*}
\mathrm{Ent}^{\Psi}_{p}(\zeta) :=\int_{\mathbb R^d} \Psi(\zeta(x)) p(x) \diff x -\Psi \left(\int_{\mathbb R^d} \zeta(x) p(x) \diff x \right).
\end{equation*}
\end{definition}

Some examples of the divergences are $\Psi: \mathbb R \rightarrow \mathbb R$ with $\Psi(x)=x^2$ (Poincar{\'e} inequality) and $\Psi: \mathbb R_{+} \rightarrow \mathbb R$ with $\Psi(x) =x\log x$ (log-Sobolev inequality).

\begin{theorem}[$\Psi$-Sobolev inequalities] \label{application-thm-1}
 Let Assumptions \ref{condn:well-defined}, \ref{condn:semi-log-convex} and \ref{condn:semi-log-concave} hold.
\begin{itemize}
    \item[(1)] 
    Let $\zeta: \mathbb R^d \rightarrow \mathcal I$ be any continuously differentiable function such that $\int_{\mathbb R^d} \zeta^2(x) p(x) \diff x \in \mathcal I$.
\begin{itemize}
    \item[(i)] If $\kappa D^2 \ge 1$, then
\begin{equation*}
    \mathrm{Ent}^{\Psi}_{p}(\zeta) \le \frac{1}{2\kappa} \int_{\mathbb R^d} \Psi''(\zeta(x)) |\nabla \zeta(x)|^2 p(x) \diff x.
\end{equation*}
    \item[(ii)] If $\kappa D^2 <1$, then
\begin{align*}
     \mathrm{Ent}^{\Psi}_{p}(\zeta) 
     \le \frac{\exp(1-\kappa D^2)}{2} D^2 \int_{\mathbb R^d} \Psi''(\zeta(x)) |\nabla \zeta(x)|^2 p(x) \diff x.
\end{align*}
\end{itemize}

\item[(2)] Fix a probability measure $\rho$ on $\mathbb R^d $ supported on a ball of radius $R$ and let $p:=N(a,\Sigma) * \rho$ and
denote $\lambda_{\min} :=\lambda_{\min}(\Sigma)$ and $\lambda_{\max}=\lambda_{\max}(\Sigma)$. Then for any continuously differentiable function $\zeta: \mathbb R^d \rightarrow \mathcal I$ such that $\int_{\mathbb R^d} \zeta^2(x) p(x) \diff x \in \mathcal I $, we have
\begin{align*}
     \mathrm{Ent}^{\Psi}_{p}(\zeta) 
     \le \frac{1}{2} \lambda_{\max} \exp\left(\frac{R^2}{\lambda_{\min}} \right) \int_{\mathbb R^d} \Psi''(\zeta(x)) |\nabla \zeta(x)|^2 p(x) \diff x.
\end{align*}

\end{itemize}

\end{theorem}

\begin{theorem}[Isoperimetric inequalities] \label{application-thm-2}
Assume that Assumptions \ref{condn:well-defined}, \ref{condn:semi-log-convex} and \ref{condn:semi-log-concave} hold. Let $\Phi$ be the cumulative distribution function of $\gamma_1$ on $\mathbb R$, that is,
\begin{equation*}
    \Phi(x)=\gamma_1(-\infty,x)=\frac{1}{\sqrt{2\pi}} \int_{-\infty}^x \exp\left(-\frac{y^2}{2} \right) \diff y, \quad  -\infty < \forall x < +\infty
\end{equation*}
and $B_2^d:=\{x\in \mathbb R^d: |x| \leq 1 \} $ be the unit ball in $\mathbb R^d$.
\begin{itemize}
    \item[(1)] Let $A_t:=A +tB_2^d$ for any Borel set $A \subseteq \mathbb R^d$ and $t \ge 0$, then
    \begin{equation*}
    p\left(A_t \right) \ge \Phi\left(p(A) +\frac{t}{C} \right), \quad 
    C :=
    \begin{cases}
            1/\sqrt{\kappa},             \ &  \text{if $\kappa D^2 \ge 1$, } \\
            \exp\left(\frac{1-\kappa D^2}{2} \right) D,  \ & \text{if $\kappa D^2 <1$.}
        \end{cases}
\end{equation*}

\item[(2)] Let $p:=N(a,\Sigma) * \rho$ where $\rho$ is a probability measure on $\mathbb R^d$ and is supported on a ball of radius $R$. Set 
$\lambda_{\min} :=\lambda_{\min}(\Sigma), \lambda_{\max} :=\lambda_{\max}(\Sigma)$ and
\begin{equation*}
    C := (\lambda_{\min} \lambda_{\max})^{1/2} \exp\left(\frac{R^2}{2\lambda_{\min}} \right).
\end{equation*}
Then
\begin{equation*}
p(A_t) \ge \Phi\left(p(A) +\frac{t}{C} \right), \quad \quad A_t:=A +tB_2^d.
\end{equation*}
\end{itemize}

\end{theorem}

Finally, let $\eta: \mathbb R^d \rightarrow \mathbb R$ be any continuously differentiable function such that
$
\int_{\mathbb R^d} \eta(x) p(x) \, \diff x =0.
$ 

\begin{theorem}[$q$-Poincar{\'e} inequalities]
\label{thm-q-poincare}
Suppose that Assumptions \ref{condn:well-defined}, \ref{condn:semi-log-convex} and \ref{condn:semi-log-concave} hold.
\begin{itemize}
    \item[(1)] Let $q \geq 2 $ be an even integer and $\eta \in L^q(\gamma_d)$, then it holds that
    \begin{equation*}
   \int_{\mathbb R^d} \eta^q(x) p(x) \, \diff x \leq \left( 
 \int_{\mathbb R^d} |\nabla \eta(x)|^q p(x) \, \diff x \right)
    \begin{cases}
     C^{\star}_1   & \text{if $\kappa D^2 \geq 1$}, \\
       C^{\star}_2  & \text{if $\kappa D^2 <1$.}
    \end{cases}
    \end{equation*}
where
\begin{equation*}
C^{\star}_1 := \left( \frac{q-1}{\kappa} \right)^{q/2}, \quad C^{\star}_2 :=D^q \exp \left(\frac{q(1-\kappa D^2)}{2} \right).
\end{equation*}

\item[(2)] Fix a probability measure $\rho$ on $\mathbb R^d$ supported on a ball of radius $R$, and let $p:= N(a,\Sigma) * \rho$ and denote $\lambda_{\min} :=\lambda_{\min}(\Sigma) $ and $\lambda_{\max} :=\lambda_{\max}(\Sigma)$. Then for any $\eta \in L^q(\gamma_d)$ with even integer $q \geq 2$, it holds that
\begin{equation*}
 \int_{\mathbb R^d} \eta^q(x) p(x) \, \diff x \le (q-1)^{\frac q 2} (\lambda_{\min} \lambda_{\max})^{\frac q 2} \exp\left(\frac{q R^2}{2 \lambda_{\min}} \right)  \int_{\mathbb R^d} |\nabla \eta(x)|^q p(x) \, \diff x.
\end{equation*}
\end{itemize}
\end{theorem}

\subsection{Non-asymptotic bounds for empirical measures}

Let 
 $\mu$ be a probability distribution on $\mathbb R^d$ and 
\begin{equation}\label{eq-empirical-measure}
    \mu_n := \frac{1}{n} \sum_{i=1}^n \delta_{X_i},
\end{equation}
be the empirical measure, where 
 $(X_i)_{i = 1}^n$ are i.i.d. samples drawn from $\mu$.
 Deriving the non-asymptotic convergence rate under the Wasserstein distance of the empirical measure $\mu_n$ and the probability measure $\mu$ on Polish space is one of the most important topics in statistics, probability, and machine learning. In recent years, significant progress has been made on this topic.
 When $p =1$, the Kantorovich-Rubinstein duality \cite{kantorovich1958space} implies that $W_1(\mu_n,\mu)$ is equivalent to the supremum of the empirical process indexed by Lipschitz functions. As a consequence, \cite{dudley1969speed} provides sharp lower and upper bounds of $\mathbb E \left[W_1(\mu_n,\mu) \right]$ for $\mu$ supported on a bounded finite dimensional set. Subsequently, \cite{talagrand1994transportation} studies the case when $\mu$ is the uniform distribution on a $d$-dimensional unit cube. For general distributions, \cite{boissard2014mean,dereich2013constructive,fournier2015rate} establish sharp upper bounds of $\mathbb E\left[W_p(\mu_n,\mu) \right]$ in finite dimensional Euclidean spaces. Recently, by extending finite dimensional spaces to infinite dimensional functional spaces, \cite{lei2020convergence} establishes similar results for general distributions. 
 
 Besides the above mentioned bounds in expectation, 
   \cite{weed2019sharp}  obtains a high probability bound on $W_p(\mu_n,\mu) $  for measures
   $\mu$ with bounded supports.
 By applying Sanov's theorem to independent random variables, \cite{bolley2007quantitative} establishes concentration inequalities for empirical measures on non-compact space.
 In this subsection, we will give  a high probability bound on $W_2(\mu_n,\mu)$ by the 
 Lipschitz transport properties proved in Theorems \ref{main-thm2} and \ref{main-thm3}.
To begin with, we will review the transportation inequality defined in Definition \ref{eq-def-transport}, the non-asymptotic convergence rate of $\mathbb E\left[ W_p(\mu_n,\mu) \right]$ as given in Theorem \ref{thm-empirical-bound}, and its  concentration inequality for $p=2$ as stated in Theorem \ref{thm-concentration-ineq}. 
Then, we will derive the  non-asymptotic convergence rate for $W_p(\mu_n,\mu)$, as stated in Theorem \ref{thm-main-concentration-ineq} 
by combining  the transportation inequality of the Gaussian measure on $\mathbb R^d$ as established in \cite{talagrand1994transportation} and the transportation inequality of the push-forward measure of the Gaussian measure under Lipschitz mapping, as shown in Lemma \ref{lem-lip-function}.

\begin{definition}[Transportation inequality]\label{eq-def-transport}
The probability measure $\mu$ satisfies the $L^p$-transportation inequality on $ \mathbb R^d $ if there is some constant $C>0$ such that for any probability measure $\nu$, $W_p(\mu,\nu) \leq \sqrt{2C H(\nu \mid \mu)}$.
To be short, we write $\mu \in \mathrm{T_p}(C)$ for this relation.
\end{definition}

\begin{lemma}[\cite{djellout2004transportation}] \label{lem-lip-function}
Assume that $\mu \in \mathrm{T_p} (C)$ on $ \mathbb R^d $. If $\Phi: \mathbb R^d \rightarrow \mathbb R^d $ is Lipschitz continuous with constant $\alpha > 0$, 
then $\nu =\mu \circ \Phi^{-1} \in \mathrm{T_p}(\alpha^2 C)$ on $ \mathbb R^d $.
\end{lemma}


\begin{theorem}[\cite{fournier2015rate}] \label{thm-empirical-bound}
    Let $p>0$, assume that for some $r>p$ and $\int_{\mathbb R^d} |x|^r \, \mu(\diff x)$ is finite. Then there exists a constant $C >0$ depending only on $p,r,d$ such that for all $n \geq 1$,
    \begin{align*}
      \mathbb E \left[ W_p(\mu_n, \mu) \right] 
       \leq C \left( \int_{\mathbb R^d} |x|^r \, \mu(\diff x)  \right)^{p/r} 
       \begin{cases}
        n^{-\frac 1 2} + n^{-\frac{r-p}{r} }, & \text{if $p> d/2$ and $r \neq 2p$ } \\
      n^{-\frac 1 2 } \log(1+n) + n^{-\frac{r-p}{r} }, & \text{if $p=d/2$ and $r \neq 2p$ } \\
      n^{-\frac p d} +n^{-\frac{r-p}{r} }, & \text{if $p < d/2$ and $r \neq \frac{d}{d-p}$ }
      \end{cases}
    \end{align*}
    where the expectation is taken on the samples $X_1, \cdots, X_n$.
\end{theorem}


The next result states that a $\mathrm{T_2}(C)$ inequality on $\mu$ implies Gaussian concentration inequality for $W_2(\mu_n,\mu)$. 

\begin{theorem}[\cite{gozlan2007large}] \label{thm-concentration-ineq}
Let a probability measure $\mu$ on $\mathbb R^d$ satisfy the transportation inequality $\mathrm{T_2}(C)$. The following holds:
\begin{equation*}
    \mathbb P\left( W_2(\mu_n,\mu) \geq \mathbb E\left[ W_2(\mu_n,\mu) \right] +t \right) \leq \exp\left(-\frac{nt^2}{C} \right).
\end{equation*}

\end{theorem}

For any probability measure $\nu$ on $\mathbb R^d$ with a finite fifth moment, let us define
\begin{align}\label{eq-mathsf-const-M}
    \mathsf M(\nu, d, n)  := c_d \left( \int_{\mathbb R^d} |x|^5 \, \nu(\diff x) \right)^{2/5} 
        \begin{cases}
         n^{-1/2} & \text{if $d<4$} \\
         n^{-1/2} \log(1+n) & \text{if $d=4$ } \\
         n^{-2/d} & \text{if $d>4$ }
        \end{cases}
\end{align}
where the constant $c_d$ depends only on $d$. On the other hand, for the $L^2$-transportation inequality $\mathrm{T_2}(C)$, recall that Talagrand \cite{talagrand1996transportation} proved that the standard Gaussian measure $\gamma_1= N(0,1)$ satisfies $\mathrm{T_2}(C)$ on $\mathbb R$ w.r.t. the Euclidean distance with the sharp constant $C = 1$ and found that $ \mathrm{T_2}(C) $ is stable for product (or independent) tensorization.  Therefore, combining Lemma \ref{lem-lip-function}, Theorems \ref{main-thm2}, \ref{main-thm3}, \ref{thm-empirical-bound} and \ref{thm-concentration-ineq}, we obtain the following results.

\begin{theorem}[Concentration for empirical measures] \label{thm-main-concentration-ineq}
Suppose that Assumptions \ref{condn:well-defined}, \ref{condn:semi-log-convex} and \ref{condn:semi-log-concave} hold, and let probability measure $\nu$ has a finite fifth moment. 
\begin{itemize}
    \item[(1)] If $\kappa D^2 \geq 1$, then $\nu \in \mathrm{T_2} (1/\kappa)$. Moreover, for any $\varepsilon \in (0,1)$, it holds that
    \begin{align*}
         W_2(\nu_n,\nu) 
          \leq  \left( \frac{\log \varepsilon^{-1} }{n \kappa} \right)^{1/2} + \mathsf{M}(\nu,d,n)
     \end{align*}
with probability at least $1-\varepsilon$ and constant $\mathsf M(\nu, d, n)$ given in (\ref{eq-mathsf-const-M}).

   \item[(2)] If $\kappa D^2 < 1$, then $\nu \in \mathrm{T_2} \left( D^2 \exp(1-\kappa D^2) \right) $. Moreover, for any $\varepsilon \in (0,1)$, it holds that
    \begin{align*}
        W_2(\nu_n,\nu) 
          \leq  \left\{ \frac{\log \varepsilon^{-1} }{nD^2 \exp(1-\kappa D^2) } \right\}^{1/2} + \mathsf M(\nu, d, n)
    \end{align*}
 with probability at least $1-\varepsilon $ and constant $\mathsf M(\nu, d, n)$ given in (\ref{eq-mathsf-const-M}).

    \item[(3)] If $\nu = N(0,\sigma^2 \mathbf{I}_d) * \rho$  where $\rho$ is a probability measure supported on a ball of radius $R$ on $\mathbb R^d$, then $\nu \in \mathrm{T_2} \left( \sigma^2 \exp(R^2/\sigma^2) \right)$. Moreover, for any $\varepsilon \in (0,1) $, it holds that
    \begin{align*}
         W_2(\nu_n,\nu)  
         \leq \left\{ \frac{\log \varepsilon^{-1} }{n\sigma^2 \exp(R^2/\sigma^2) } \right\}^{1/2}
         + \mathsf M(\nu, d, n)
    \end{align*}
  with  probability at least $1-\varepsilon $ and constant $\mathsf M(\nu, d, n)$ given in (\ref{eq-mathsf-const-M}).
\end{itemize}

\end{theorem}


\section{Conclusion}
We have constructed the F{\"o}llmer flow originating from a standard Gaussian measure and hitting a general target measure.
By studying the well-posedness of the F{\"o}llmer flow, we have established the Lipschitz property of its flow map at time $t = 1$. 
Such a Lipschitz transport map enables get functional  inequalities with dimension-free constants  and derive concentration inequalities for the empirical measure for 
 rich classes of probability measures.
It is worthwhile to notice that the F{\"o}llmer velocity field has an analytic expression that is compatible with Monte Carlo approximations. 
Therefore, a possible direction of future research would be to design general-purpose sampling algorithms and score-based generative models using the F{\"o}llmer flow.
Besides, being limited to scenarios covered in Assumptions \ref{condn:semi-log-convex} and \ref{condn:semi-log-concave}, the work could be extended  to explore weaker and even minimal regularity assumptions on the target measure. For example, replacing semi-log-concavity with ``convexity at infinity" in \cite{bolley2012convergence,cattiaux2014semi} is a potential step.    

\newpage 

\begin{appendix}

\section{Proof of Theorem \ref{main-thm1} and Proposition \ref{prop:vp-ode-eps}}

\subsection{Well-definedness of the F{\"o}llmer flow}
Recall that the velocity field $V(t, x)$ defined in (\ref{eq:vector-field}) yields 
\begin{equation*}
V(t, x) :=\frac{\nabla \log \EuScript{Q}_{1-t} r(x)}{t}, \quad r(x):=\frac{p(x)}{\phi(x)}
\end{equation*}
where $\nu(\diff x)=p(x) \diff x$. For any $t \in (0,1]$, then one obtains
\begin{align*}
    \EuScript{Q}_{1-t} r (x) 
   = \int_{\mathbb R^d} \varphi^{tx, 1-t^2} (y) r(y) \diff y 
    = \int_{\mathbb{R}^d} \phi (z) r (tx + \sqrt{1 - t^2}z) \diff z,
\end{align*}
where $\varphi^{tx, 1-t^2} (y)$ is the density of the $d$-dimensional Gaussian measure with mean $ tx$ and covariance $(1-t^2)\mathbf{I}_d $. For the convenience of subsequent calculation, we introduce the following symbols:
\begin{equation*}
S(t, x) :=\nabla \log q_t(x), \quad q_t(x) :=\int_{\mathbb R^d} q(t,x|1,y) p(y) \diff y
\end{equation*}
where $q(t,x|1,y) :=(2\pi (1-t^2))^{-\frac d 2} \exp\left(-\frac{|x-ty|^2}{2(1-t^2)} \right)$ for any $t\in [0,1]$. 
Notice that
\begin{align*}
 \EuScript{Q}_{1-t} r (x)
  = (2\pi)^{d/2} \exp\left(\frac{ |x |^2}{2} \right) \times  \frac{1}{(2\pi(1-t^2))^{d/2}} \int_{\mathbb R^d} p(y)\exp\left(-\frac{|x-ty |^2 }{2(1-t^2)} \right) \diff y.
\end{align*}
Then we have $\nabla \log \EuScript{Q}_{1-t} r(x) = x + S(t, x)$, for any $t \in [0,1]$.

Suppose that the target distribution $p$ satisfies the third moment condition, we can supplement the definition of velocity field $V$ at time $t = 0$, so that $V$ is well-defined on the interval $[0,1]$. Then we have the following result:
\begin{lemma}\label{lem-4.1}
Suppose that $\mathbb E_{p}[|X|^3] <\infty$, then
\begin{equation*}
\lim_{t \downarrow 0} V(t, x) =\lim_{t \downarrow 0} \frac{x +S(t, x)}{t} =\mathbb E_{p}[X].
\end{equation*}
\end{lemma}

\begin{proof}
Let $t \to 0$, then it yields
\begin{align*}
\lim_{t \downarrow 0} V(t, x) =\lim_{t \downarrow 0} \partial_t S(t,x) 
= \lim_{t \downarrow 0} \left \lbrace \frac{\nabla[\partial_{t} q_t(x)]} {q_t(x)}- \frac{\partial_{t} q_t(x)} {q_t(x)} S(t, x) \right \rbrace.
\end{align*}
On the one hand, by simple calculation, it holds that
\begin{align*}
    \partial_{t} q_t(x) 
    & = \partial_{t} \int_{\mathbb R^d} q(t, x | 1, y) p(y) \diff y = \partial_{t} \int_{\mathbb R^d} \left[ 2 \pi ( 1 - t^2 ) \right]^{-\frac{d}{2}} \exp \left( - \frac{|x - ty|^2} { 2 (1 - t^2) } \right) p(y) \diff y \\
    &= \frac { t d} {1 - t^2} q_t(x) - \frac{t} {(1 - t^2)^2} |x|^2 q_t(x) + \frac{1 + t^2} {(1 - t^2)^2} \int_{\mathbb R^d} x^{\top} y q(t,x | 1, y) p(y) \diff y \\
    & ~~~ - \frac{t} {(1 - t^2)^2} \int_{\mathbb R^d} |y|^2 q(t,x | 1,y) p(y) \diff y.
\end{align*}
Furthermore, we also obtain
\begin{align*}
       \frac{\partial_{t} q_t(x)} {q_t(x)}
       = \frac {t d} {1 - t^2}
       - \frac{t} {(1 - t^2)^2} |x|^2 + \frac{1 + t^2} {(1 - t^2)^2} \int_{\mathbb R^d} x^{\top} y q(1,y | t,x) \diff y 
       -\frac{t} {(1 - t^2)^2} \int_{\mathbb R^d} |y|^2 q(1,y | t,x) \diff y.
\end{align*}
On the other hand, by straightforward calculation, it yields
\begin{align*}
       \nabla [\partial_{t} q_t(x)]
    &= - \frac { td} {(1 - t^2)^2} x q_t(x) + \frac{t^2 d} {(1 - t^2)^2} \int_{\mathbb R^d} y q(t,x | 1,y) p(y) \diff y \\
    & \quad - \frac{2 t} {(1 - t^2)^2} x q_t(x) + \frac{t} {(1 - t^2)^3}  |x|^2 x q_t(x) - \frac{t^2 |x|^2} {(1 - t^2)^3}  \int_{\mathbb R^d} y q(t,x | 1,y) p(y) \diff y \\
    & \qquad + \frac{1 + t^2} {(1 - t^2)^2} \int_{\mathbb R^d} y q(t,x | 1,y) p(y) \diff y \\
    &\quad - \frac{1 + t^2} {(1 - t^2)^3} \int_{\mathbb R^d} (x^{\top} y)x q(t,x | 1,y) p(y) \diff y
     + \frac{t (1 + t^2)} {(1 - t^2)^3} \int_{\mathbb R^d} (x^{\top} y)y q(t,x | 1,y) p(y) \diff y\\
    & \quad + \frac{t} {(1 - t^2)^3}  \int_{\mathbb R^d} x |y|^2 q(t,x| 1,y) p(y) \diff y - \frac{t^2} {(1 - t^2)^3} \int_{\mathbb R^d} y |y|^2 q(t,x | 1,y) p(y) \diff y.
\end{align*}
Moreover, we also obtain
\begin{align*}
       \frac{\nabla [\partial_{t} q_t(x)]} {q_t(x)}
    & = -\frac {t d } {(1 - t^2)^2} x + \frac{t^2 d} {(1 - t^2)^2} \int_{\mathbb R^d} y q(1,y | t,x) \diff y - \frac{2 tx} {(1 - t^2)^2}   + \frac{t |x|^2 x} {(1 - t^2)^3}  \\
    & \quad  - \frac{t^2} {(1 - t^2)^3} |x|^2 \int_{\mathbb R^d} y q(1,y | t,x) \diff y + \frac{1 + t^2} {(1 - t^2)^2} \int_{\mathbb R^d} y q(1,y | t,x) \diff y \\
    & \quad - \frac{1 + t^2} {(1 - t^2)^3} \int_{\mathbb R^d} (x^{\top} y)x q(1,y | t,x) \diff y
     + \frac{t (1 + t^2)} {(1 - t^2)^3} \int_{\mathbb R^d} (x^{\top} y)y q(1,y | t,x) \diff y \\
    & \quad  + \frac{t} {(1 - t^2)^3} x \int_{\mathbb R^d} |y|^2 q(1,y | t,x) \diff y
     - \frac{t^2} {(1 - t^2)^3} \int_{\mathbb R^d} y |y|^2 q(1,y | t,x) \diff y.
\end{align*}
Since $\mathbb E_p[|X|^3] <\infty$, it yields
\begin{equation*}
   \lim_{t \downarrow 0}  \int_{\mathbb R^d} |y|^3 q(1, y | t, x) \diff y = \int_{\mathbb R^d} |y|^3 \lim_{t \downarrow 0} q(1, y | t, x) \diff y =\mathbb E_p \left[|X|^3 \right] < + \infty.
\end{equation*}
Furthermore, we have
\begin{align*}
       \lim_{t \downarrow 0} \frac{\partial_{t} q_t(x)} {q_t(x)} S(t, x) = - x x^{\top} \E_{p} [X], \quad \lim_{t \downarrow 0} \frac{\nabla[\partial_{t} q_t(x)]} {q_t(x)}
    &= \E_{p} [X] - x x^{\top} \E_{p} [X].
\end{align*}
Therefore, it yields $\lim_{t \downarrow 0} V(t, x) = \E_{p} [X]$, which completes the proof.
\end{proof}

\subsection{Cram{\'e}r-Rao inequality}
In order to obtain a lower bound of the Jacobian matrix of velocity field $V$ defined in (\ref{eq:vector-field}), we apply the classical Cram{\'e}r-Rao bound \cite{rao1945information, cramer1946mathematical} in statistical parameter estimation to a special case for location parameter estimation. This particular application is far from being new in information theory and convex geometry (for example, see \cite{dembo1991information, lutwak2002cramer, cianchi2013unified, saumard2014log}). We include it here for the sake of completeness. 
A lower bound on the covariance matrix of a prescribed probability measure directly follows the Cram{\'e}r-Rao bound \cite[Theorem 11.10.1]{cover2005chapter11}.

\begin{lemma}[Cram{\'e}r-Rao bound]
\label{lem:cr-ineq}
Let $\mu_{\theta}(\diff x) = f_{\theta}(x) \diff x$ be a probability measure on $\sR^d$ such that the density $f_{\theta}(x)$ is of class $C^2$ with respect to an unknown parameter $\theta \in \Theta$. Assume that a few mild regularity assumptions hold. Then provided that $\{ X_i \}_{i=1}^n$ are i.i.d. samples from $\mu_{\theta}$ with size $n$,
the mean-squared error of any unbiased estimator $g(X_1, X_2, \cdots, X_n)$ for the parameter $\theta$ is lower bounded by the inverse of the Fisher information matrix:
\begin{align*}
     \E_{\mu_{\theta}} \left[ (g(X_1, X_2, \cdots, X_n) - \theta)^{\otimes 2} \right] \succeq \left( - n \E_{\mu_{\theta}} \left[ \frac{\partial^2} {\partial \theta^2} \log f_{\theta}(X_1) \right] \right)^{-1}.
\end{align*}
\end{lemma}

We consider the example of location parameter estimation. Suppose $\theta$ is the location parameter and let $f_{\theta} (x) = f(x - \theta)$ and $g(x) = x$. Specifically, it yields a lower bound on the covariance matrix of the probability measure $\mu_{\theta}$ in the case that $\theta = \E_{\mu_{\theta}} [X]$, i.e., a random sample $X \sim \mu_{\theta}$ is an unbiased estimator of the mean $\theta$. Apart from this implication, an alternative proof of the same lower bound on the covariance matrix is presented in \cite{chewi2022entropic}. It is worth noting that a compactly supported probability measure $\mu$ would suffice to ensure the Cram{\'e}r-Rao inequality holds.

\begin{lemma}
\label{lem-cramer-rao}
Let $\mu(\diff x) = \exp(-U(x)) \diff x$ be a probability measure on $\sR^d$ such that $U$ of class $C^2$ on the interior of its domain. Suppose $X$ is a random sample from $\mu$. Then the covariance matrix is lower bounded as
$\mathrm{Cov}_{\mu} (X) \succeq \left( \mathbb E_{\mu} \left[ \nabla^2 U(X) \right] \right)^{-1}$.
\end{lemma}

\subsection{Proof of Propositions  \ref{prop:vp-ode-eps}}

\begin{proof}
By It{\^ o} SDE defined in (\ref{eq:vp-sde-tr}), we have the distribution $\overline{X}_t$, which is given by 
\begin{equation} \label{eq-sp-x}
\overline{X}_t|\overline{X}_0=x_0 \sim N((1-t)x_0, t(2-t) \mathbf{I}_d).
\end{equation}
Due to the Cauchy-Lipschitz theory \cite[Section 2]{ambrosio2014continuity}, the push-forward map $X^{*}_t$ and process $\overline{X}_t$ have the same distribution for any $t\in [0,1-\varepsilon]$. Then by (\ref{eq-sp-x}), we obtain
\begin{equation*}
X^{*}_t \overset{d}{=} \overline{X}_{t} \overset{d}{=} (1-t) X + \sqrt{t(2-t)} Y
\end{equation*}
with $X \sim \nu, Y \sim \gamma_d$. Recall that $X^{*}_{1 - \varepsilon}$ and $\varepsilon X + \sqrt{1-\varepsilon^2} Y$ have the same distribution. Therefore, by the definition of $W_2$ and Cauchy-Schwarz's inequality, it yields
\begin{align*}
 W^2_2 (\nu \circ ({X^{*}_{1 - \varepsilon}})^{-1}, \gamma_d) 
& \leq \int_{\mathbb R^d \times \mathbb R^d} |\varepsilon x +(\sqrt{1-\varepsilon^2} -1) y |^2 p(x) \phi(y) \diff x \diff y  \notag \\
& \leq  2 \varepsilon^2 \int_{\mathbb R^d} |x|^2 p(x) \diff x + 2 \left(\sqrt{1-\varepsilon^2} -1 \right)^2 \int_{\mathbb R^d} |y|^2 \phi(y) \diff y  \notag \\
&= 2 \varepsilon^2 \, \mathbb E_p[|X|^2] + 2d\left(\sqrt{1-\varepsilon^2} -1 \right)^2.
\end{align*}
Let $\varepsilon \rightarrow 0$, and it yields
$\lim_{\varepsilon \to 0} W_2(\nu \circ ({X^{*}_{1 - \varepsilon}})^{-1}, \gamma_d) = 0$, which completes the proof.
\end{proof}


\section{Proof of Theorem \ref{main-thm2} and Theorem \ref{main-thm3}}

\subsection{Bound on the Lipschitz constant of the flow map}

We take a close look at Lipschitz properties of the F{\"o}llmer flow \paref{main-ode}.
In order to derive functional inequalities, we deploy the approach of Lipschitz changes of variables from the Gaussian measure $\gamma_d$ to the target measure $\nu$.

One key argument is to bound the maximum eigenvalue of the Jacobian matrix of velocity field denoted as $\lambda_{\max} (\nabla V(t, x))$.
By integrating both sides of \paref{main-ode} w.r.t. time $s \in [0, t]$, we have
\begin{equation}
    \label{eq:vp-ode-integral}
    X_t(x) - X_0(x) = \int_0^t V(s, X_s(x)) \diff s, \quad X_0(x)=x.
\end{equation}
Taking the first-order derivative w.r.t $x$ on both sides of (\ref{eq:vp-ode-integral}), we get
\begin{equation}
    \label{eq:vp-map-gradx}
    \nabla X_t(x) - \nabla X_0(x) = \int_0^t \nabla  V(s, X_s(x)) \nabla X_s(x) \diff s.
\end{equation}
Taking the first-order derivative w.r.t $t$ on both sides of (\ref{eq:vp-map-gradx}), we get
\begin{equation}
    \label{eq:vp-map-gradx-gradt}
    \frac{\partial} {\partial t} \nabla X_t(x) = \nabla V(t, X_t(x)) \nabla X_t(x).
\end{equation}
Let $a_t = | \nabla X_t(x)  r |^2 $ with $|r|=1$. Assume $\lambda_{\max} (\nabla V(t, x)) \le \theta_t$. By (\ref{eq:vp-map-gradx-gradt}) we get 
\begin{align*}
    \frac{\partial} {\partial t} a_t  = 2 \left<(\nabla X_t)r, \frac{\partial}{\partial t} (\nabla X_t) r \right> 
     = 2 \left<(\nabla X_t)r, \left(\nabla V(t, X_t) \nabla X_t\right) r \right>  \leq 2\theta_t a_t.
\end{align*}
The above display and Gr{\"o}nwall's inequality imply
\begin{align*}
\|\nabla X_t (x) \|_{\mathrm{op}}=\sup_{\|r \|_2=1} \sqrt{a_t} 
 \le \sup_{\|r \|_2=1} \sqrt{a_0} \exp \left( \int_0^t \theta_s \diff s \right)  = \exp \left( \int_0^t \theta_s \diff s \right).
\end{align*}
Let $t = 1$, then we get
\begin{equation}
    \label{eq:vp-map-lips}
    \mathrm{Lip}(X_1(x)) \le \|\nabla X_1 (x) \|_{\mathrm{op}} \le \exp \left( \int_0^1 \theta_s \diff s \right).
\end{equation}

\subsection{Lipschitz properties of transport maps}

In this subsection, we show that the considered flow map is Lipschitz in various settings. The following is the main result of this subsection and it covers the Lipschitz statements of Theorem \ref{main-thm2} and Theorem \ref{main-thm3}.

\begin{theorem} \label{section-5-thm2}
\begin{itemize}
    \item[(1)] Suppose that either $p$ is $\kappa$-semi-log-concave for some $\kappa >0$, or $p$ is $\kappa$-semi-log-concave for some $\kappa \in \mathbb R$ and $D <+ \infty$. Then the F{\"o}llmer flow \paref{main-ode} has a unique solution for all $t\in [0,1]$. Furthermore,
\begin{itemize}
    \item[(a)] If $\kappa D^2 \ge 1 $, then $X_1(x)$ is a Lipschitz mapping with constant $\tfrac{1}{\sqrt{\kappa}}$, or equivalently,
    \begin{equation*}
    \|\nabla X_1(x) \|^2_{\mathrm{op}} \le \frac{1}{\kappa}, \quad \forall x \in \mathbb{R}^d.
    \end{equation*}

    \item[(b)] If $\kappa D^2 <1 $, then $X_1(x)$ is a Lipschitz mapping with constant $\exp\left(\frac{1-\kappa D^2}{2} \right)D $, or equivalently,
    \begin{equation*}
    \|\nabla X_1(x) \|^2_{\mathrm{op}} \le \exp\left( 1-\kappa D^2 \right)D^2, \quad \forall x \in \mathbb{R}^d.
    \end{equation*}
\end{itemize}

\item[(2)] Fix a probability measure $\rho$ on $\mathbb R^d $ supported on a ball of radius $R$ and let $p :=N(0,\sigma^2 \mathbf{I}_d) * \rho $. Then the F{\"o}llmer flow \paref{main-ode} has a unique solution for all $t\in [0,1]$. Furthermore, $X_1(x)$ is a Lipschitz mapping with constant $\sigma \exp \left( \frac{R^2}{2 \sigma^2} \right)$, or equivalently,
\begin{equation*}
    \|\nabla X_1(x) \|^2_{\mathrm{op}} \le \sigma^2 \exp\left( \frac{R^2}{\sigma^2} \right), \quad \forall x \in \mathbb{R}^d.
    \end{equation*}
\end{itemize}
\end{theorem}

\begin{proof}
Combine Theorem \ref{thm1}-(4) and Corollaries \ref{cor1}-\ref{cor2} and then complete the proof.
\end{proof}

In fact, the existence of a solution to the IVP (\ref{main-ode}) also relies on controlling $\nabla V$. To this end, we represent $\nabla V$ as a covariance matrix. We start by defining a measure $p^{tx, 1-t^2}$ on $\sR^d$, for fixed $t \in [0, 1)$ and $x \in \sR^d$, by
\begin{align}
    \label{eq:measure-cov-vp}
    p^{tx, 1-t^2} (y)
    :=& \frac{ \varphi^{tx, 1-t^2} (y) r (y) } {\EuScript{Q}_{1-t} r (x)}, \quad r := \frac{\diff \nu} {\diff \gamma_d} = \frac{p} {\phi},
\end{align}
where $\varphi^{tx, 1-t^2} (y)$ is the density of the $d$-dimensional Gaussian measure with mean $ tx$ and covariance $(1-t^2)\mathbf{I}_d $ and
\begin{align}\label{eq:vp-vf-op-form}
    \EuScript{Q}_{1-t} r (x) 
     = \int_{\mathbb R^d} \varphi^{tx, 1-t^2} (y) r(y) \diff y = \int_{\mathbb{R}^d} \phi (z) r (tx + \sqrt{1 - t^2}z) \diff z.
\end{align}
Notice that
\begin{align*}
  \EuScript{Q}_{1-t} r (x)= (2\pi)^{d/2} \exp\left(\frac{ |x |^2}{2} \right) \frac{1}{(2\pi(1-t^2))^{d/2}} \int_{\mathbb R^d} p(y)\exp\left(-\frac{|x-ty |^2 }{2(1-t^2)} \right) \diff y.
\end{align*}
Hence, we obtain
\begin{equation*}
 V(t, x) =\frac{x + S(t, x)}{t} = \frac{\nabla \log \EuScript{Q}_{1-t} r (x)}{t}, \quad 0< t \le 1.
\end{equation*}

\begin{lemma}\label{lem-cov}
Suppose velocity field $V$ is defined in (\ref{eq:vector-field}), then
\begin{equation}
    \label{eq:cov-vp-vf}
    \nabla V(t, x) = \frac{t} {(1 - t^2)^2} \Cov \left( p^{tx, 1-t^2} \right) - \frac{t} {1 - t^2} \rmI_d, \quad \forall t\in (0,1), \quad \nabla  V(0, x) =0.
\end{equation}
\end{lemma}


\begin{proof}
By taking the first-order and the second-order derivatives on both sides in (\ref{eq:vp-vf-op-form}), we get
\begin{align*}
    \nabla \EuScript{Q}_{1-t} r (x) &= \frac{t} {1 - t^2} \int_{\mathbb R^d} (y - tx) \varphi^{tx, 1-t^2} (y) r(y) \diff y, \\
    \nabla^2 \EuScript{Q}_{1-t} r (x) &= \frac{t^2} { (1 - t^2)^2 } \int_{\mathbb R^d} (y - tx)^{\otimes 2} \varphi^{tx, 1-t^2} (y) r(y) \diff y - \left( \frac{t^2} {1 - t^2} \int_{\mathbb R^d} \varphi^{tx, 1-t^2} (y) r(y) \diff y \right) \rmI_d.
\end{align*}
Then we obtain
\begin{align*}
     & \nabla^2 \log \EuScript{Q}_{1-t} r (x) \\
    =&\frac{ \nabla^2 \EuScript{Q}_{1-t} r (x) } { \EuScript{Q}_{1-t} r (x) } - \left( \frac{ \nabla \EuScript{Q}_{1-t} r (x) } { \EuScript{Q}_{1-t} r (x) } \right)^{\otimes 2} \\
    =& \frac{t^2} { (1 - t^2)^2 } \left[ \int_{\mathbb R^d} (y - tx)^{\otimes 2} p^{tx, 1-t^2} (y) \diff y - \left( \int_{\mathbb R^d} (y - tx) p^{tx, 1-t^2} (y) \diff y \right)^{\otimes 2} \right] - \frac{t^2} {1 - t^2} \rmI_d \\
    =& \frac{t^2} { (1 - t^2)^2 } \left[ \int_{\mathbb R^d} y^{\otimes 2} p^{tx, 1-t^2} (y) \diff y - \left( \int_{\mathbb R^d} y p^{tx, 1-t^2} (y) \diff y \right)^{\otimes 2} \right] - \frac{t^2} {1 - t^2} \rmI_d\\
    =& \frac{t^2} { (1 - t^2)^2 } \Cov (p^{tx, 1-t^2} ) - \frac{t^2} {1 - t^2} \rmI_d.
\end{align*}
Therefore, we get
\begin{equation}
    \label{eq:cov-vp-score}
    \nabla V(t, x) = \frac{t} {( 1 - t^2 )^2} \Cov ( p^{tx, 1-t^2} ) - \frac{t} {1 - t^2} \rmI_d.
\end{equation}
This completes the proof.
\end{proof}

Next, we use the representation of (\ref{eq:cov-vp-vf}) to estimate the upper bound of $\nabla V(t, x)$.
\begin{theorem}\label{thm1}
Let $p$ be a probability measure on $\mathbb R^d$ with $D := (1/\sqrt{2} ) \mathrm{diam} (\mathrm{supp} (p))$.
\begin{itemize}
    \item[(1)] For every $t \in [0, 1)$,
        \begin{equation}
            \label{eq:vp-vf-ubd-bounded}
        \frac{t}{1-t^2} \mathbf{I}_d \preceq  \nabla V(t, x) \preceq \left( \frac{t D^2} {(1 - t^2)^2} - \frac{t} {1 - t^2} \right) \rmI_d.
        \end{equation}

    \item[(2)] Suppose that $p$ is $\beta$-semi-log-convex with $\beta \in (0, +\infty)$. Then for any $t\in [0,1]$,
   \begin{equation}
       \nabla V(t, x) \succeq \frac{t (1 - \beta)} {\beta (1 - t^2) + t^2} \rmI_d.
   \end{equation}
     In particular, when $p \sim N\left(0, \frac{1}{\beta} \mathbf{I}_d \right)$, then 
   \begin{equation*}
       \nabla V(t, x) =\frac{t (1 - \beta)} {\beta (1 - t^2) + t^2} \rmI_d.
   \end{equation*}
    
    \item[(3)] Let $\kappa \in \sR$ and suppose that $p$ is $\kappa$-semi-log-concave. Then for any $t \in \left[ \sqrt{ \frac{\kappa} {\kappa - 1} \Id_{\kappa < 0} }, 1 \right]$,
        \begin{equation}
            \label{eq:vp-vf-ubd-log-concave}
            \nabla V(t, x) \preceq \frac{t (1 - \kappa)} {\kappa (1 - t^2) + t^2} \rmI_d.
        \end{equation}

    \item[(4)] Fix a probability measure $\rho$ on $\sR^d$ supported on a ball of radius $R$ and let $p := N(0,\sigma^2 \mathbf{I}_d) * \rho$ with $\sigma > 0$. Then for any $t \in [0, 1]$,
    \begin{align}\label{eq:vp-vf-ubd-convolution}
        \frac{(\sigma^2-1) t}{1+ (\sigma^2-1)t^2} \mathbf{I}_d   \preceq  \nabla V(t, x)  \preceq t \left\{ \frac{ (\sigma^2 - 1) [ 1 + (\sigma^2 - 1) t^2 ] + R^2 } { [ 1 + (\sigma^2 - 1) t^2 ]^2 } \right\} \rmI_d.
        \end{align}
\end{itemize}
\end{theorem}

\begin{proof}
The proof idea of this theorem follows similar arguments as in \cite[Lemma 3.3]{mikulincer2021brownian}.
\begin{itemize}
    \item[(1)] By \cite[Theorem 2.6]{danzer1963helly}, there exists a closed ball with radius less than $D := (1/\sqrt{2} ) \mathrm{diam} (\mathrm{supp} (p))$ that contains $\mathrm{supp} (p)$ in $\sR^d$.
    Then the desired bounds are a direct result of $0 \rmI_d \preceq \Cov (p^{tx, 1-t^2}) \preceq D^2 \rmI_d$ and (\ref{eq:cov-vp-vf}).

 \item[(2)] For any $t\in (0,1)$, recall that (\ref{eq:cov-vp-vf}) reads
\begin{equation} \label{eq-cov-v}
    \nabla V(t, x) = \frac{t} {(1 - t^2)^2} \Cov \left( p^{tx, 1-t^2} \right) - \frac{t} {1 - t^2} \rmI_d.
\end{equation}
On the one hand, let $p$ be $\beta$-semi-log-convex for some $\beta >0$. Then for any $t\in [0,1), p^{tx, 1-t^2}$ is $\left(\beta +\frac{t^2}{1-t^2} \right)$-semi-log-convex because
\begin{align*}
 -\nabla^2 \log \left( p^{tx,1-t^2}(y) \right) 
 =-\nabla^2 \log \left(r(y) \phi(y) \right) -\nabla^2 \log \left(\frac{\varphi^{tx,1-t^2}(y)}{\phi(y) } \right) 
 \preceq \left(\beta +\frac{t^2}{1-t^2} \right) \mathbf{I}_d
\end{align*}
where we use that $ p(y) = r(y) \phi(y)$. On the other hand, by Lemma \ref{lem-cramer-rao}, we obtain
\begin{equation*}
\mathrm{Cov}\left(p^{tx,1-t^2} \right) \succeq \left(\beta +\frac{t^2}{1-t^2} \right)^{-1} \rmI_d.
\end{equation*}
Furthermore, by (\ref{eq-cov-v}), we obtain
\begin{align*}
     \nabla V(t, x)  \succeq \left \lbrace \frac{t}{(1-t^2)^2} \left(\beta +\frac{t^2}{1-t^2} \right)^{-1} -\frac{t}{1-t^2} \right \rbrace \mathbf{I}_d 
     =\frac{t(1-\beta)}{\beta (1-t^2) +t^2} \mathbf{I}_d.
\end{align*}
Recall that $(X_t)_{t \in [0, 1]}$ satisfies the IVP (\ref{main-ode}), then we have
\begin{equation*}
    \nabla V(t, x) = \frac{\nabla^2 \log  \EuScript{Q}_{1-t} r(x)}{t}, \quad r(x) :=\frac{p(x)}{\phi(x)}.
\end{equation*}
Since $p \sim N\left(0,\frac{1}{\beta} \mathbf{I}_d \right)$, then it yields
\begin{equation*}
    r(x) = \beta^{d/2} \exp\left(-\frac{\beta -1}{2} |x|^2 \right) \propto \exp\left(-\frac{\beta -1}{2} |x|^2 \right),
\end{equation*}
where the symbol $\propto$ signifies equality up to a constant which does not depend on $x$. Then by straightforward calculation for $\EuScript{Q}_{1-t} r(x)$, we obtain
\begin{align*}
 \EuScript{Q}_{1-t} r(x) 
& \propto \int_{\mathbb R^d} \exp \left \lbrace -\frac{\beta -1}{2} \left|tx +\sqrt{1-t^2} y \right|^2 -\frac{|y|^2}{2} \right \rbrace \diff y \\
 &= \int_{\mathbb R^d} \exp \left \lbrace -\frac{(\beta -1)t^2}{2}|x|^2 -(\beta -1)t \sqrt{1-t^2} \left<x,y \right>   -\frac{\beta(1-t^2) +t^2}{2} |y|^2 \right \rbrace \diff y \\
 &= \exp\left(-\frac{(\beta-1)t^2 |x|^2 }{2\beta_t} \right) \int_{\mathbb R^d} \exp\left \lbrace -\frac{\beta_t}{2} \left|y +\frac{(\beta -1)t \sqrt{1-t^2} }{\beta_t} x \right|^2 \right \rbrace \diff y,
 \end{align*}
where we denote $\beta_t :=(1-t^2) \beta +t^2 $. Considering that the integrand in the last line is proportional to the density of a Gaussian measure, then the value of the integral does not depend on $x$, and
\begin{align*}
    \EuScript{Q}_{1-t} r(x) 
    \propto \exp\left(-\frac{(\beta-1)t^2 |x|^2 }{2\beta_t} \right)  =\exp\left(-\frac{|x|^2}{2} \cdot \frac{(\beta -1)t^2}{(1-t^2)\beta +t^2} \right).
\end{align*}
So we have
\begin{equation*}
    \nabla V(t, x) =\frac{\nabla^2 \EuScript{Q}_{1-t} r(x)}{t} =\frac{(1-\beta) t}{(1-t^2) \beta +t^2} \mathbf{I}_d.
\end{equation*}

    \item[(3)] Let $p$ be $\kappa$-semi-log-concave. Then for any $t\in [0,1)$, $p^{tx, 1-t^2}$ is $\left(\kappa +\frac{t^2}{1-t^2} \right)$-semi-log-concave because
\begin{align*}
 -\nabla^2 \log \left( p^{tx,1-t^2} (y) \right) 
 =-\nabla^2 \log \left(r(y) \phi(y) \right) -\nabla^2 \log \left(\frac{\varphi^{tx,1-t^2}(y)}{\phi(y) } \right) 
 \succeq \left(\kappa +\frac{t^2}{1-t^2} \right) \mathbf{I}_d
\end{align*}
where we use $ p(y) =r(y) \phi(y)$. If $t\in \left[ \sqrt{ \frac{\kappa} {\kappa - 1} \Id_{\kappa < 0} }, 1 \right]$, then $\kappa +\frac{t^2}{1-t^2} \ge 0$. By the well-known Brascamp-Lieb inequality \cite{bakry2014analysis,brascamp1976extensions}, applied to functions of the form $\mathbb R^d \ni x \mapsto f(x)= \left<x, v \right>$ for any $ v \in \mathbb S^{d-1}$, we obtain
\begin{equation*}
\mathrm{Cov}\left(p^{tx,1-t^2} \right) \preceq \left(\kappa +\frac{t^2}{1-t^2} \right)^{-1} \rmI_d
\end{equation*}
and the result follows by (\ref{eq:cov-vp-vf}).

\item[(4)] On the one hand, we have
\begin{align*}
  p^{tx,1-t^2}(y) 
  =\frac{(N(0,\sigma^2 \mathbf{I}_d) * \rho)(y)}{\varphi^{0,1}(y)} \cdot \frac{\varphi^{tx,1-t^2}(y)}{\EuScript{Q}_{1-t} \left( \frac{N(0,\sigma^2 \mathbf{I}_d) *\rho )}{\varphi^{0,1}} \right) (x) }  =A_{x,t} \int_{\mathbb R^d} \varphi^{z,\sigma^2}(y) \varphi^{\frac{x}{t}, \frac{1-t^2}{t^2}} (y) \rho(\diff z),
\end{align*}
where the constant $A_{x,t}$ depends only on $x$ and $t$. Moreover, we obtain
\begin{equation*}
\ p^{tx,1-t^2}(y)  =\int_{\mathbb R^d} \varphi^{\frac{(1-t^2)z +\sigma^2 tx}{1+(\sigma^2-1)t^2},\frac{\sigma^2(1-t^2)}{1+(\sigma^2-1)t^2}}(y) \tilde \rho(\diff z)
\end{equation*}
where $\tilde \rho$ is a probability measure on $\mathbb R^d$ which is a multiple of $\rho$ by a positive function. In particular, $\tilde \rho$ is supported on the same ball as $\rho$. On the other hand, let $Z \sim \gamma_d$ and $Y \sim \tilde{\rho}$ be independent. Then
\begin{align*}
     \sqrt{\frac{\sigma^2(1-t^2)}{1+(\sigma^2-1)t^2}} Z +\frac{(1-t^2)}{1+(\sigma^2-1)t^2}Y  +\frac{t \sigma^2}{1+(\sigma^2-1) t^2} x \sim p^{tx,1-t^2}.
\end{align*}
Due to $0 \rmI_d \preceq \mathrm{Cov}(Y) \preceq R^2 \rmI_d$, it holds that
\begin{align*}
\frac{\sigma^2 (1-t^2)}{1 + (\sigma^2-1) t^2} \mathbf{I}_d \preceq \mathrm{Cov}(p^{tx,1-t^2}) 
\preceq \frac{\sigma^2(1-t^2)[1+(\sigma^2-1)t^2] + (1-t^2)^2 R^2}{[1+(\sigma^2-1)t^2]^2} \mathbf{I}_d.
\end{align*}
By applying (\ref{eq:cov-vp-vf}) again, it yields
\begin{align*}
 \frac{(\sigma^2 -1)t}{1+(\sigma^2 -1)t^2} \mathbf{I}_d \preceq \nabla V(t, x) \preceq  t \left\{ \frac{ (\sigma^2 - 1) [ 1 + (\sigma^2 - 1) t^2 ] + R^2 } { [ 1 + (\sigma^2 - 1) t^2 ]^2 } \right\} \rmI_d.
\end{align*}
\end{itemize}
This completes the proof of Theorem \ref{thm1}.
\end{proof}

Next, we present an upper bound on $\lambda_{\max} (\nabla  V(t, x))$ and its exponential estimation.

\begin{corollary}\label{cor1}
Let $p$ be a probability measure on $\mathbb R^d$ with $D :=  (1/\sqrt{2}) \mathrm{diam} ( \mathrm{supp} (p))$ and suppose that $p$ is $\kappa$-semi-log-concave with $\kappa \in [0, +\infty)$.
\begin{itemize}
    \item[(1)] If $\kappa D^2 \ge 1$, then
        \begin{equation}
            \label{eq:vp-max-egv-ubd-ge1}
            \lambda_{\max} (\nabla V(t, x)) \le \theta_t := \frac{t (1 - \kappa)} {t^2 (1 - \kappa) + \kappa}.
        \end{equation}
        and
        \begin{equation}
            \label{eq:vp-max-egv-ubd-exp-ge1-time-1}
            \exp \left( \int_{0}^1 \theta_s \diff s \right)
            = \frac{1} {\sqrt{\kappa}}.
        \end{equation}

    \item[(2)] If $\kappa D^2 < 1$, then
        \begin{equation}
            \label{eq:vp-max-egv-ubd-less1}
            \lambda_{\max} (\nabla V(t, x))  \le \theta_t :=
                \begin{cases}
                    \frac{t (t^2 + D^2 - 1)} {(1 - t^2)^2},             \ &t \in [0, t_0], \\
                    \frac{t (1 - \kappa)} {t^2 (1 - \kappa) + \kappa},  \ &t \in [t_0, 1],
                \end{cases}
        \end{equation}
        where $t_0 = \sqrt{ \frac{1 - \kappa D^2} {(1 - \kappa) D^2 + 1}}$ and
        \begin{equation}
            \label{eq:vp-max-egv-ubd-exp-less1}
            \exp \left( \int_{0}^1 \theta_s \diff s \right)
            = \exp \left( \frac{1-\kappa D^2}{2} \right) D.
        \end{equation}
\end{itemize}
\end{corollary}

\begin{proof}
By Theorem \ref{thm1}, we obtain
\begin{align*}
 \lambda_{\max}(\nabla V(t, x)) \le \frac{tD^2}{(1-t^2)^2} -\frac{t}{1-t^2}, \quad \lambda_{\max}(\nabla V(t, x)) \le \frac{t(1-\kappa)}{\kappa(1-t^2) +t^2}, \quad \forall t\in [0,1].
\end{align*}
By simple algebra calculation, it yields
\begin{align*}
 \frac{t(D^2 +t^2-1)}{(1-t^2)^2} \le \frac{t(1-\kappa)}{\kappa(1-t^2)+t^2}  \quad \text{if and only if} \quad (1+D^2-\kappa D^2)t^2 \le 1-\kappa D^2.
\end{align*}
We consider two cases.
\begin{itemize}
    \item[(1)] $\kappa D^2 \ge 1$: By considering $\kappa D^2=1$, we see that the bound $(1+D^2-\kappa D^2)t^2 \le 1-\kappa D^2 $ cannot hold. So it would be advantageous to use the bound
\begin{align*}
\lambda_{\max}(\nabla V(t, x)) 
\le \theta_t := \frac{t(1-\kappa)}{\kappa(1-t^2) +t^2} 
= \frac{t(1-\kappa)}{t^2(1-\kappa) +\kappa}.
\end{align*}
Next, we will compute $\exp\left(\int_0^1 \theta_t \diff t\right)$ and we first check that the integral $\int_0^1 \theta_t \diff t$ is well-defined. For this reason, we only need to consider whether the sign of the denominator $(1-\kappa)t^2 +\kappa $ is equal to 0.

The only case is $(1-\kappa)t^2 +\kappa =0$ that happens when $t^2_0:= \kappa/(\kappa -1)$. If $\kappa \in (0,1]$, $(1-\kappa) t^2 +\kappa \ne 0$. Thus, $\theta_t$ is integrable on $[0,1]$. If $\kappa >1$, $t_0 >1$. Then $\theta_t$ is integrable on $[0,1]$ as well. The only case is $\kappa = 0$ which results in $t_0 =0$. However, in this case, we cannot have $\kappa D^2 \ge 1$ as $\kappa =0$. Then by simple calculation,
\begin{align*}
 \int_0^1 \theta_t \diff t =(1-\kappa)\int_0^1 \frac{t \diff t}{(1-\kappa)t^2 +\kappa} =-\frac{1}{2}\log \kappa, 
 \quad \exp\left(\int_0^1 \theta_t \diff t \right) \le \frac{1}{\sqrt{\kappa}}.
\end{align*}

\item[(2)] $\kappa D^2 <1$: The condition $(1+D^2 -\kappa D^2) t^2 \le 1-\kappa D^2$ is equivalent to
\begin{equation*}
    t \le \sqrt{\frac{1-\kappa D^2}{1+(1-\kappa) D^2}}
\end{equation*}
since the denominator is nonnegative as $\kappa D^2 <1$. Hence, we define
\begin{equation*}
            \lambda_{\max} (\nabla V(t, x)) \le \theta_t :=
                \begin{cases}
                    \tfrac{t (t^2 + D^2 - 1)} {(1 - t^2)^2},             \ & 0\le t \le t_0 , \\
                    \tfrac{t (1 - \kappa)} {t^2 (1 - \kappa) + \kappa},  \ &  t_0 \le  t \le 1,
                \end{cases}
        \end{equation*}
 where $t_0 := \sqrt{ \tfrac{1 - \kappa D^2} {(1 - \kappa) D^2 + 1}} $. In order to compute integral $\int_0^1 \theta_t \diff t $, we note that, following the discussion in the case $\kappa D^2 \ge 1 $, the denominators $1-t^2$ and $(1-\kappa )t^2 +\kappa $ do not vanish in the intervals $[0,t_0]$ and $[t_0,1]$, respectively. For $t \in [0, t_0]$, using integral by parts, we have
\begin{align*}
& \int_0^{t_0} \frac{t(t^2 +D^2 -1)}{(1-t^2)^2} \diff t  =\frac{1}{2} \int_0^{t_0} (t^2+D^2-1) \diff \left(\frac{1}{1-t^2} \right) \\
& =\left. \frac{t^2+D^2-1}{2(1-t^2)} \right|^{t=t_0}_{t=0} -\frac{1}{2} \int_0^{t_0} \frac{2t}{1-t^2} \diff t = \frac{t^2_0}{2(1-t^2_0)} D^2 +\frac{1}{2} \log(1-t^2_0) \\
& =\frac{1-\kappa D^2}{2} +\frac{1}{2} \log \left(\frac{D^2}{1+(1-\kappa) D^2} \right).
\end{align*}
For $t\in [t_0, 1]$, we have
\begin{align*}
\int_{t_0}^1 \frac{t(1-\kappa)}{\kappa +(1-\kappa)t^2} \diff t =  -\frac{1}{2} \log \left(t^2_0 +(1-t^2_0) \kappa \right) =-\frac{1}{2} \log \left(\frac{1}{1+(1-\kappa) D^2} \right).
\end{align*}
Hence, we obtain
\begin{equation*}
\int_0^1 \theta_t \diff t=\int_0^{t_0} \theta_t \diff t +\int_{t_0}^1 \theta_t \diff t =\frac{1-\kappa D^2}{2} +\log D.
\end{equation*}
Then
\begin{equation*}
\exp\left(\int_0^1 \theta_t \diff t \right) =\exp\left(\frac{1-\kappa D^2}{2} +\log D \right) =D \exp\left(\frac{1-\kappa D^2}{2} \right).
\end{equation*}
\end{itemize}
This completes the proof of Corollary \ref{cor1}.
\end{proof}

\begin{corollary}\label{cor2}
Let $p$ be a probability measure on $\mathbb R^d$ with $D :=  (1/\sqrt{2}) \mathrm{diam} ( \mathrm{supp} (p)) < \infty$ and suppose that $p$ is $\kappa$-semi-log-concave with $\kappa \in (-\infty, 0)$. We have
\begin{equation}
    \label{eq:vp-max-egv-ubd-kappa-nega}
    \lambda_{\max} (\nabla V(t, x)) \le \theta_t :=
        \begin{cases}
            \tfrac{t (t^2 + D^2 - 1)} {(1 - t^2)^2},             \ &t \in [0, t_0] \\
            \tfrac{t (1 - \kappa)} {t^2 (1 - \kappa) + \kappa},  \ &t \in [t_0, 1]
        \end{cases}
\end{equation}
where $t_0 = \sqrt{ \tfrac{1 - \kappa D^2} {(1 - \kappa) D^2 + 1} }$ and
\begin{equation}
    \label{eq:vp-max-egv-ubd-exp-kappa-nega}
    \exp \left( \int_{0}^1 \theta_s \diff s \right)
    = \exp \left( \frac{1-\kappa D^2}{2} \right) D.
\end{equation}
\end{corollary}

\begin{proof}
By Theorem \ref{thm1}, we obtain
\begin{align*}
\lambda_{\max}(\nabla V(t, x))  \le \frac{tD^2}{(1-t^2)^2} -\frac{t}{1-t^2},  \quad \forall t\in [0, 1), \quad
 \lambda_{\max}(\nabla V(t, x)) \le \frac{t(1-\kappa)}{\kappa(1-t^2) +t^2}, \quad \forall t \in \left[\sqrt{\frac{\kappa}{\kappa -1}}, 1 \right].
\end{align*}
Then it yields
\begin{equation*}
\lambda_{\max}(\nabla V(t, x)) \le \frac{t(t^2 +D^2-1)}{(1-t^2)^2}, \quad \forall t\in \left[0, \sqrt{\frac{\kappa}{\kappa -1}} \right).
\end{equation*}
Next, since $0<\sqrt{\tfrac{\kappa}{\kappa -1}} <\sqrt{\tfrac{1-\kappa D^2}{(1-\kappa)D^2 +1}} \le 1 $ and $ \kappa(1-t^2) +t^2 \ge 0 $ for all $t \ge \sqrt{\frac{\kappa}{\kappa -1}}$, then one obtains
\begin{equation*}
\frac{t(t^2 +D^2 -1)}{(1-t^2)^2} \le  \frac{t(1-\kappa)}{\kappa(1-t^2) +t^2}
\end{equation*}
for all $t\in \left[\sqrt{\frac{\kappa}{\kappa -1}}, \sqrt{\frac{-\kappa D^2}{(1-\kappa)D^2 +1}} \right]$. We define
\begin{equation*}
    \lambda_{\max} (\nabla  V(t, x)) \le \theta_t :=
        \begin{cases}
            \tfrac{t (t^2 + D^2 - 1)} {(1 - t^2)^2},             \ &t \in [0, t_0] \\
            \tfrac{t (1 - \kappa)} {t^2 (1 - \kappa) + \kappa},  \ &t \in [t_0, 1]
        \end{cases}
\end{equation*}
where $t_0 := \sqrt{ \tfrac{1 - \kappa D^2} {(1 - \kappa) D^2 + 1} }$. As in the proof of Corollary \ref{cor1}, it holds that
\begin{align*}
\int_0^{t_0} \theta_t \diff t  = \frac{1-\kappa D^2}{2} +\frac{1}{2} \log \left(\frac{D^2}{1+(1-\kappa) D^2} \right), \quad
\int_{t_0}^1 \theta_t \diff t  =-\frac{1}{2} \log \left(\frac{1}{1+(1-\kappa) D^2} \right).
\end{align*}
Then we have
\begin{align*}
\int_0^1 \theta_t \diff t  =\frac{1-\kappa D^2}{2} +\log D, \quad
\exp\left(\int_0^1 \theta_t \diff t \right)  = D \exp\left(\frac{1-\kappa D^2}{2} \right).
\end{align*}
This completes the proof of Corollary \ref{cor2}.
\end{proof}


\section{Proof of Theorems \ref{application-thm-1}, \ref{application-thm-2} and \ref{thm-q-poincare} }

We start with a differential Lipschitz mapping $T: \mathbb R^d \rightarrow \mathbb R^d$ associated with constant $C$. The following result describes the Lipschitz properties of the derivatives of composite mappings.

\begin{lemma} \label{lem1}
Let $T: \mathbb R^d \rightarrow \mathbb R^d $ be a differential Lipschitz mapping with constant $C$ and let $\zeta: \mathbb R^d \rightarrow \mathbb R$ be a continuously differentiable function. Then
\begin{equation*}
\nabla (\zeta \circ T) = [(\nabla \zeta) \circ T ] \nabla T,
\end{equation*}
where $(\nabla T)(x) : \mathbb R^d \rightarrow \mathbb R^{d\times d}$ be a Jacobian matrix for any $x\in \mathbb R^d$. Furthermore, we obtain
\begin{align*}
|\nabla \zeta(T(x))| 
   \le \|\nabla T(x) \|_{\mathrm{op}} \cdot |(\nabla \zeta) \circ (T(x)) | \le C |(\nabla \zeta) \circ (T(x)) |
\end{align*}
for all $ x \in \mathbb R^d $.
\end{lemma}

Since the proof of this result is almost trivial by using the chain rule and the Lipschitz mapping $T$, we omit it here. Through Lemma \ref{lem1}, we can start the proofs of the functional inequalities which follow from Theorems \ref{main-thm2} and \ref{main-thm3}. We first begin with the $\Psi$-Sobolev inequalities defined in \cite{chafai2004entropies}.

\subsection{Proof of Theorem \ref{application-thm-1} }
\begin{proof}
\begin{itemize}
    \item[(1)] It can be seen from \cite[Corollary 2.1]{chafai2004entropies} that for standard Gaussian measure $\gamma_d$ on $\mathbb R^d$, we have the following $\Psi$-Sobolev inequalities:
    \begin{equation}\label{eq:Sobolev ineq}
    \mathrm{Ent}^{\Psi}_{\gamma_d}(F) \le \frac{1}{2} \int_{\mathbb R^d} \Psi''(F) |\nabla F|^2 \diff \gamma_d
    \end{equation}
for any smooth function $F: \mathbb R^d \rightarrow \mathcal I $. Let $(X_t)_{t\in [0,1]}$ be the solution of IVP (\ref{main-ode}) so that $X_1 \sim p $ if $X_0 \sim N(0,\mathbf{I}_d)$. Suppose that $X_1(x): \mathbb R^d \rightarrow \mathbb R^d$ is a Lipschitz mapping with constant $C$ and let $F :=\zeta \circ X_1: \mathbb R^d \rightarrow \mathcal I$ with $\zeta: \mathbb R^d \rightarrow \mathcal I$. Then combining Lemma \ref{lem1}, (\ref{eq:Sobolev ineq}) and $p= \gamma_d \circ (X_1)^{-1}$ we have
\begin{align*}
\mathrm{Ent}^{\Psi}_{p}(\zeta) 
 =\mathrm{Ent}^{\Psi}_{\gamma_d}(F) 
 \le \frac{1}{2} \int_{\mathbb R^d} \Psi''(F) |\nabla F|^2 \diff \gamma_d 
& \le \frac{C^2}{2} \int_{\mathbb R^d} \Psi''(\zeta \circ X_1) |\nabla \zeta \circ X_1|^2 \diff \gamma_d \\
& = \frac{C^2}{2} \int_{\mathbb R^d} \Psi''(\zeta(x)) |\nabla \zeta(x)|^2 p(x) \diff x.
\end{align*}
The proof is complete by Theorems \ref{main-thm2} and \ref{main-thm3}.

\item[(2)] Let random vector $Y \sim \rho$, let $\tilde \rho$ be the law of $\Sigma^{-1/2} Y$, and define $\tilde p:=\gamma_d * \tilde \rho $. Set $\lambda_{\min} :=\lambda_{\min}(\Sigma)$ and $\lambda_{\max} :=\lambda_{\max}(\Sigma)$. Then combining Lemma \ref{lem1}, (\ref{eq:Sobolev ineq}) and $\tilde p= \gamma_d \circ (X_1)^{-1}$, we have
\begin{equation*}
\mathrm{Ent}^{\Psi}_{\tilde p}(\zeta) \le \frac{\exp(\lambda^{-1}_{\min} R^2)}{2} \int_{\mathbb R^d} \Psi''(\zeta(x)) |\nabla \zeta(x)|^2 \tilde p(x) \diff x.
\end{equation*}
Let $p=N(a,\Sigma) * \rho$ and let $\tilde X \sim \tilde p$ such that
\begin{align*}
\Sigma^{1/2} \tilde X +a 
 =\Sigma^{1/2}\left(X+\Sigma^{-1/2} Y \right) +a 
 =\left( \Sigma^{1/2} X +a \right)+Y \sim p=N(a,\Sigma) * \rho,
\end{align*}
where $X\sim N(0,\mathbf{I}_d)$. Given $\zeta: \mathbb R^d \rightarrow \mathcal I$ and let $\tilde \zeta(x) :=\zeta(\Sigma^{1/2} x+ a)$ so that
\begin{align*}
\mathrm{Ent}^{\Psi}_{p}(\zeta) 
 =\mathrm{Ent}^{\Psi}_{\tilde p}(\tilde \zeta) 
\le \frac{\exp(\lambda^{-1}_{\min} R^2)}{2} \int_{\mathbb R^d} \Psi''(\tilde \zeta(x)) |\nabla \tilde \zeta(x)|^2 \tilde p(x) \diff x.
\end{align*}
Since $(\nabla \tilde \zeta)(x)=\Sigma^{1/2} \left(\nabla \zeta(\Sigma^{1/2}x +a) \right)$, we get
\begin{equation*}
|(\nabla \tilde \zeta)(x)|^2 \le \lambda_{\max} \left|\left(\nabla \zeta (\Sigma^{1/2}x +a) \right) \right|^2.
\end{equation*}
Furthermore, it yields that
\begin{align*}
 \mathrm{Ent}^{\Psi}_{p}(\zeta)  \le \frac{\lambda_{\max} \exp(\lambda^{-1}_{\min} R^2)}{2} \int_{\mathbb R^d} \Psi''(\zeta(x)) |\nabla \ \zeta(x)|^2 p(x) \diff x.
\end{align*}
\end{itemize}
This completes the proof.
\end{proof}

\subsection{Proof of Theorem \ref{application-thm-2} }

\begin{proof}
\begin{itemize}
\item[(1)] 
By using \cite[Theorem 4.3]{ledoux1996isoperimetry}, then the Gaussian measure $\gamma_d$ on $\mathbb R^d$ satisfies the following Gaussian isoperimetric inequality:
\begin{equation*}
    \gamma_d (K_t) \ge \Phi\left(\gamma_d(K) +t \right), \quad t \geq 0
\end{equation*}
for any Borel measurable set $K_t :=K+t B_2^d$ and $K \subseteq \mathbb R^d$. Therefore, suppose $(X_t)_{t\in [0,1]}$ be the solution of IVP (\ref{main-ode}) so that $X_1 \sim p $ if $X_0 \sim N(0,\mathbf{I}_d)$. Moreover, suppose that $X_1(x): \mathbb R^d \rightarrow \mathbb R^d$ is a Lipschitz mapping with constant $C$, then for any fixed $ x \in \mathbb R^d$,
\begin{equation*}
 |X_1(x+ y) -X_1(x)| \le C |y|, \quad \forall y \in \mathbb R^d.
\end{equation*}
We first show the following result:
\begin{equation}
\label{eq:iso-inequality}
X^{-1}_1(E) +\frac{t}{C} B_2^d \subseteq X^{-1}_1(E_t), \quad E_t :=E +t B_2^d
\end{equation}
for any Borel measurable set $E \subseteq \mathbb R^d$ and $t\geq 0$. 
To obtain (\ref{eq:iso-inequality}), we only need to prove that
\begin{equation*}
    X_1\left(X^{-1}_1(E) +\frac{r}{C} B_2^d \right) \subseteq E_t, \quad t\geq 0
\end{equation*}
or, in other words, if $x \in X^{-1}_1(E) +\frac{t}{C} B_2^d$, then $X_1(x) \in E_t$ for any Borel measurable set $K$. Furthermore, if we assume
\begin{equation*}
 x \in X^{-1}_1(E) +\frac{t}{C} B_2^d \quad \text{so that} \quad  x=\theta +\frac{t}{C} h
\end{equation*}
for some $\theta \in X^{-1}_1(E)$ and $ h \in B_2^d$, we have $X_1\left(x -\frac{t}{C} h \right) \in E$. Then it yields that
\begin{equation*}
\left|X_1\left( x-\frac{t}{C} h \right) -X_1(x) \right| \leq t, \quad t\geq 0
\end{equation*}
where $ x -\frac{t}{C} h \in X^{-1}_1(E)$. Therefore, $X_1(x) \in E_t$ as desired. Finally, combining the Gaussian isoperimetric inequality and (\ref{eq:iso-inequality}), it yields
\begin{align*}
 p(E_t) =\gamma_d\left(X^{-1}_1(E_t) \right) 
  \ge \gamma_d\left(X^{-1}_1(E) +\frac{t}{C} B_2^d \right) 
  \geq \Phi\left(\gamma_d\left[X^{-1}_1(E) +\frac{t}{C} \right] \right) 
  =\Phi\left(p(E)+ \frac{t}{C} \right).
\end{align*}
This proof is completed by Theorems \ref{main-thm2} and \ref{main-thm3}.

\item[(2)] Let random vector $Y \sim \rho$, let $\tilde \rho$ be the law of $\Sigma^{-1/2} Y$, and define measure $\tilde p:= \gamma_d * \tilde \rho$. Set $\lambda_{\min} :=\lambda_{\min}(\Sigma)$ and $\lambda_{\max} :=\lambda_{\max}(\Sigma)$. Similar to the argument of part (1), for any Borel set $E \subset \mathbb R^d$ and $t \ge 0$, we obtain
\begin{equation*}
  \tilde p\left( E_t \right) \ge \Phi\left(\tilde p(E) + \frac{t}{C}\right), C:= \left( \lambda_{\min} \right)^{1/2} \exp\left(\frac{R^2}{2\lambda_{\min}} \right).
\end{equation*}
Let $p=N(a,\Sigma) * \rho$ and let $\tilde X \sim \tilde p$ so that \begin{equation*}
\Sigma^{1/2} \tilde X +a =\left( \Sigma^{1/2} X +a \right)+Y \sim p=N(a,\Sigma) * \rho
\end{equation*}
and $X\sim N(0,\mathbf{I}_d)$. Then for any Borel measurable set $E \subset \mathbb R^d$ and $t \ge 0$, it yields that
\begin{align*}
    p(E_t) =\tilde p \left(\Sigma^{-1/2}(E-a) +\Sigma^{-1/2} tB_2^d \right)  \ge \tilde p\left(\Sigma^{-1/2} (E-a) +t \lambda^{-1/2}_{\max} B_2^d \right).
\end{align*}
Hence, we obtain
\begin{equation*}
p(E_t) \ge \Phi\left(\tilde p\left[\Sigma^{-1/2} (E-a) \right] +\frac{t \lambda^{-1/2}_{\max}}{C} \right).
\end{equation*}
We obtain the desired result by applying $\tilde p\left[\Sigma^{-1/2}(E-a) \right] = p(E)$.
\end{itemize}
This completes the proof.
\end{proof}

\subsection{Proof of Theorem \ref{thm-q-poincare} }

\begin{proof}
\begin{itemize}
\item[(1)] We will use the fact that \cite[Proposition 3.1]{nourdin2009second} the $q$-Poincar{\'e} inequality holds for the standard Gaussian measure $\gamma_d$ on $\mathbb R^d$:
\begin{equation}\label{eq:Poincare ineq}
\mathbb E_{\gamma_d} \left[F^q \right] \le (q-1)^{q/2} \mathbb E_{\gamma_d} \left[|\nabla F|^q \right],
\end{equation}
for any smooth function $F \in L^q(\gamma_d)$ with $\mathbb E_{\gamma_d}[F] =0$. Let $(X_t)_{t\in [0,1]}$ be the solution of IVP \paref{main-ode} so that $X_1 \sim p $ if $X_0 \sim N(0,\mathbf{I}_d)$. Suppose that $X_1(x): \mathbb R^d \rightarrow \mathbb R^d$ is a Lipschitz mapping with constant $C$ and let $F :=\eta \circ X_1$. Then combining Lemma \ref{lem1}, (\ref{eq:Poincare ineq}) and $p= \gamma_d \circ (X_1)^{-1}$ we have
\begin{align*}
\mathbb E_p[\eta^q] =\mathbb E_{\gamma_d}[F^q] 
 \leq (q-1)^{q/2} \mathbb E_{\gamma_d}[|\nabla F|^q] \leq  C^q (q-1)^{q/2} \mathbb E_p[|\nabla \eta|^q].
\end{align*}
The proof is complete by Theorems \ref{main-thm2} and \ref{main-thm3}.

\item[(2)] Let $Y \sim \rho$, let $\tilde \rho$ be the law of $\Sigma^{-1/2} Y$, and define $\tilde p:=N(0,\mathbf{I}_d) * \tilde \rho$. Set $\lambda_{\min} :=\lambda_{\min}(\Sigma) $ and $\lambda_{\max} :=\lambda_{\max}(\Sigma) $. The argument of part (1) gives,
\begin{equation*}
\mathbb E_{\tilde p} [\eta^q] \le \exp\left(\frac{q R^2}{2 \lambda_{\min}} \right) \lambda_{\min}^{q/2} (q-1)^{q/2}  \mathbb E_{\tilde p}[ |\nabla \eta|^q].
\end{equation*}
Let $p=N(a,\Sigma) * \rho $ and let $\tilde X \sim \tilde p$ such that \begin{equation*}
\Sigma^{1/2} \tilde X +a =\left( \Sigma^{1/2} X +a \right)+Y \sim p=N(a,\Sigma) * \rho
\end{equation*}
and $X \sim N(0,\mathbf{I}_d)$. Let $\tilde \eta(x) :=\eta\left(\Sigma^{1/2} x +a \right) $ so that
\begin{align*}
\mathbb E_p[\eta^q ] 
 =\mathbb E_{\tilde p} \left[(\tilde \eta )^q \right]  \leq \exp\left(\frac{q R^2}{2 \lambda_{\min}} \right) \lambda_{\min}^{q/2} (q-1)^{q/2}  \mathbb E_{\tilde p}[ |\nabla \tilde \eta|^q].
\end{align*}
Since $(\nabla \tilde \eta)(x)=\Sigma^{1/2} \left( \nabla \eta\left(\Sigma^{1/2} x+ a \right) \right)$ we have
\begin{equation*}
 \left|(\nabla \tilde \eta)(x) \right|^q \le (\lambda_{\max})^{q/2} \left| \nabla \eta \left(\Sigma^{1/2} x + a \right) \right|^q.
\end{equation*}
Further, we obtain
\begin{align*}
 \mathbb E_p[\eta^q] = \mathbb E_{\tilde p} \left[(\tilde \eta)^q \right]  \leq  (\lambda_{\min} \lambda_{\max})^{q/2} \exp\left(\frac{q R^2}{2 \lambda_{\min}} \right) (q-1)^{q/2} \mathbb E_p[|\nabla \eta |^q].
\end{align*}
\end{itemize}
This completes the proof.
\end{proof}

\section{Time changes}

\begin{lemma} \label{lm:sde-time-change}
Let $(\overline{X}_t)_{t\in [0,1)}$ be a diffusion process defined by (\ref{eq:vp-sde-tr}) with $\varepsilon \to 0$ and 
let $(\overline{Y}_s)_{s \ge 0}$ be an Ornstein-Uhlenbeck process $(\overline{Y}_s)_{s \ge 0}$ defined by
\begin{equation} \label{eq:ou-sde}
\diff \overline{Y}_s = - \overline{Y}_s \diff s +\sqrt{2} \diff \overline{W}_s, \quad \overline{Y}_0 \sim \nu, \quad s \ge 0.
\end{equation}
Then $(\overline{X}_t)_{t\in [0,1)}$ is equivalent to $(\overline{Y}_s)_{s \ge 0}$ through the change of time formula $t = 1 - e^{-s}$.
\end{lemma}

\begin{proof}
Let $s = -\log(1-t)$ for any $t \in [0, 1)$. 
By applying (\ref{eq:vp-sde-tr}), it yields
\begin{equation*}
\diff \overline{X}_{1-e^{-s}} = - \overline{X}_{1-e^{-s}} \diff s +\sqrt{2} \diff \overline{W}_s, \quad \overline{X}_0 \sim \nu, \quad s \ge 0.
\end{equation*}
On the one hand, since (\ref{eq:ou-sde}) has a unique strong solution, it indicates $\overline{Y}_s = \overline{X}_{1-e^{-s}}$ for all $s \ge 0$. On the other hand, the infinitesimal generator of Markov process $(\overline{Y}_s)_{s \ge 0}$ is given by
\begin{equation}
    \label{eq:gene-ou}
    L^{\overline{Y}} = \Delta - x \cdot \nabla.
\end{equation}
By using (\ref{eq:vp-sde-tr}), the infinitesimal generator of $(\overline{X}_t)_{t\in [0,1)}$ is given by
\begin{equation}
    \label{eq:gene-vp}
    L^{\overline{X}}_t = \frac{1}{1-t} ( \Delta - x \cdot \nabla ).
\end{equation}
Furthermore, combining the chain rule and straightforward calculation, we obtain that processes $ \overline{X}_t$ and $\overline{Y}_s$ have the same infinitesimal generator, which implies $\overline{X}_t =\overline{Y}_s$ for any $t \in [0,1), s = - \log(1 - t)$.
\end{proof}

\begin{lemma} \label{lm:ode-time-change}
Let $(X^*_t)_{t\in [0,1)}$ be the time reversal of a F{\"o}llmer flow associated to probability measure $\nu$ defined by (\ref{eq:vp-ode-eps-tr}) with $\varepsilon \to 0$ and 
let $(Y^*_s)_{s \ge 0}$ be a heat flow from probability measure $\nu$ to the standard Gaussian measure $\gamma_d$ defined by
\begin{align} \label{eq:ode-heat-flow}
 \diff Y^*_s (x) = - \nabla \log \left\{ \int_{\mathbb{R}^d} r \left( e^{-s} Y^*_s (x) + \sqrt{1 - e^{-2s}} z \right) \diff \gamma_d (z) \right\}
\diff s
\end{align}
where $r(x) := (\diff \nu / \diff \gamma_d)(x), Y^*_0 \sim \nu$ for all $s \geq 0$. Then $(X^*_t)_{t\in [0,1)}$ is equivalent to $(Y^*_s)_{s \ge 0}$ through the change of time formula $t = 1 - e^{-s}$.
\end{lemma}

\begin{proof}
Let $s = -\log(1-t)$ for every $t \in [0, 1)$. 
By (\ref{eq:vp-ode-eps-tr}), it yields
\begin{align*}
 \diff X^*_{1-e^{-s}} (x) = - \nabla \log \left\{ \int_{\mathbb{R}^d} r \left( e^{-s} X^*_{1-e^{-s}} (x) + \sqrt{1 - e^{-2s}} z \right) \diff \gamma_d (z) \right\} \diff s
\end{align*}
where $X^*_0 \sim \nu $ for all $s \geq 0$. The expression above indicates that $Y^*_s := X^*_{1-e^{-s}}$ satisfies (\ref{eq:ode-heat-flow}).
\end{proof}

\end{appendix}


\bibliographystyle{plainnat}

\bibliography{Follmer_Flow}


\end{document}